\documentclass[a4paper,12pt]{amsart}

\usepackage{amssymb}
\usepackage{amsmath,a4wide,inputenc,euscript}
\usepackage{amsthm}
\usepackage[french, english]{babel}
\usepackage{graphics}
\usepackage{graphicx}

\author{Nicolae Mihalache}
\address{KTH - Royal Institute of Technology \newline
         Department of Mathematics \newline
         100 44 Stockholm, Sweden \newline \newline
         Supported by the \emph{Knut and Alice Wallenberg Foundation \newline}}
\email{nicolae@kth.se}
\title{Two counterexamples in rational and interval dynamics}
\date{}

\newtheorem*{thma}{Theorem A}
\newtheorem*{thmb}{Theorem B}

\newtheorem{theorem}{Theorem}[section]
\newtheorem{prop}[theorem]{Proposition}
\newtheorem{definition}[theorem]{Definition}
\newtheorem{lemma}[theorem]{Lemma}
\newtheorem{corollary}[theorem]{Corollary}
\newtheorem*{remark}{Remark}
\theoremstyle{remark}

\newcommand{\R}{\mathbb{R}}
\newcommand{\C}{\mathbb{C}}
\newcommand{\al}{\alpha}
\newcommand{\be}{\beta}
\newcommand{\ga}{\gamma}
\newcommand{\de}{\delta}
\newcommand{\De}{\Delta}
\newcommand{\te}{\theta}
\newcommand{\si}{\sigma}
\newcommand{\la}{\lambda}

\def\ra{{\rightarrow}}
\def\se{{\subseteq}}
\def\ve{{\varepsilon}}
\def\ez{{\varepsilon_0}}
\def\ft{{\infty}}
\def\sm{{\setminus}}
\def\es{{\emptyset}}
\def\pa{{\partial}}
\def\RS{{\overline\C}}
\def\mo{^{-1}}
\def\mfa{{\mbox{ for all }}}
\def\ma{{\mbox{ and }}}
\def\gt{{\ga_\ft}}
\def\gtp{{\ga'_\ft}}
\def\gl{{\ol\ga}}
\def\gg{{g_\gt}}

\newcommand\MO[1]{\mathop{\mathrm{#1}}\nolimits}

\newcommand\MOo[2]{\mathop{\mathrm{#1}}\nolimits\left(#2\right)}

\newcommand\Deg[2]{\deg_{#1}\left({#2}\right)}

\newcommand\Ku{\underline{K}}
\newcommand\iu{\ensuremath{\underline{i}}}
\newcommand\ku{\underline{k}}
\newcommand\um{\underline{m}}

\newcommand\SetEnu[2]{\ensuremath{\left\{{#1},\ldots,{#2}\right\}}}
\newcommand\SetDef[2]{\ensuremath{\left\{{#1}\ :\ {#2}\right\}}}
\newcommand\Set[1]{\ensuremath{\left\{{#1}\right\}}}

\newcommand\rset{\ensuremath{I\setminus\left\{c_1,\ldots,c_l\right\}}}

\newcommand\var[2]{\ensuremath{={#1},\ldots,{#2}}}
\newcommand\ol[1]{\ensuremath{\overline{#1}}}
\newcommand\sq[3]{\ensuremath{\left({#1}_{#2}\right)_{#2\geq #3}}}
\newcommand\sqno[1]{\ensuremath{\left({#1}_{n}\right)_{n\geq 1}}}
\newcommand\sqnz[1]{\ensuremath{\left({#1}_{n}\right)_{n\geq 0}}}
\newcommand\gab[1]{\ga\in[\al_{#1},\be_{#1}]}
\newcommand\gabp[1]{\ga'\in[\al_{#1}',\be_{#1}']}
\newcommand\ag[1]{{\mbox{ for all }\gab{#1}}}
\newcommand\agp[1]{{\mbox{ for all }\gabp{#1}}}
\newcommand\allg[1]{for all $\gab{#1}$}
\newcommand\allgp[1]{for all $\gabp{#1}$}
\newcommand\fd[2]{\ensuremath{:{#1}\rightarrow{#2}}}
\newcommand\limi[2]{\ensuremath{\lim\limits_{{#1}\ra\ft}{#2}}}

\newcommand\cG{{conditions (\ref{equGRep}) to (\ref{equGAl}) }}
\newcommand\IFT{{Implicit Functions Theorem }}

\newcommand\rthm[1]{{Theorem \ref{thm#1}}}

\newcommand\rpro[1]{{Proposition \ref{prop#1}}}
\newcommand\rpros[2]{{Propositions \ref{prop#1} and \ref{prop#2}}}

\newcommand\rdef[1]{{Definition \ref{def#1}}}
\newcommand\rdefs[2]{{Definitions \ref{def#1} and \ref{def#2}}}
\newcommand\rsec[1]{{Section \ref{sect#1}}}
\newcommand\rlem[1]{{Lemma \ref{lem#1}}}
\newcommand\rlems[2]{{Lemmas \ref{lem#1} and \ref{lem#2}}}
\newcommand\rcor[1]{{Corollary \ref{cor#1}}}
\newcommand\rcors[2]{{Corollaries \ref{cor#1} and \ref{cor#2}}}
\newcommand\requ[1]{{(\ref{equ#1})}}
\newcommand\requs[2]{{(\ref{equ#1}) and (\ref{equ#2})}}

\newcommand\ineq[1]{{inequality (\ref{equ#1})}}

\newcommand\ineqss[2]{{inequalities (\ref{equ#1}), (\ref{equ#2})}}

\newcommand\sit{\ensuremath{SI_2}}
\newcommand\siti{\ensuremath{SI_2^\ft}}
\newcommand\iti{\ensuremath{I_2^\ft}}
\newcommand\ts{\ensuremath{\times\Sigma}}
\newcommand\iab[1]{\ensuremath{[\al_{#1},\be_{#1}]}}
\newcommand\nm[1]{\ensuremath{\left|\left|{#1}\right|\right|}}
\newcommand\abs[1]{\ensuremath{\left|{#1}\right|}}
\newcommand\lr[1]{\ensuremath{\left({#1}\right)}}
\newcommand\nc[1]{\nm{#1}_{C^1}}

\newcommand\lm[1]{\la^{-{#1}}}
\newcommand\lmp[1]{\la^{-\lr{#1}}}

\newcommand\A{\ensuremath{\mathcal A}}
\newcommand\F{\ensuremath{\mathcal F}}
\newcommand\G{\ensuremath{\mathcal G}}

\newcommand\Hc{\ensuremath{\mathcal H}}
\newcommand\Ht{\ensuremath{\tilde{\mathcal H}}}
\newcommand\Sc{\ensuremath{\mathcal S}}
\newcommand\Pc{\ensuremath{\mathcal P}}
\def\ti{\tilde}

\newcommand\intI{\ensuremath{\stackrel{\circ}{I}}}

\def\dia{\MO{diam}}

\def\dst{\MO{dist}}
\def\cri{\MO{Crit}}

\newcommand\dw[1]{\dia W^{#1}}

\begin{document}

\maketitle
\begin{abstract}
In rational dynamics, we prove the existence of a polynomial that satisfies the \emph{Topological Collet-Eckmann 
condition}, but which has a recurrent critical orbit that is not Collet-Eckmann. This shows that the converse of the main theorem 
in \cite{MCN1} does not hold.

In interval dynamics, we show that the Collet-Eckmann property for recurrent critical orbits is not a topological invariant 
for real polynomials with negative Schwarzian derivative. This contradicts a conjecture of {\'S}wi{\c{a}}tek \cite{TRCE}.
\end{abstract}


\section{Introduction}
In one-dimensional real and complex dynamics, there are several conditions which guarantee some form of non-uniform hyperbolicity, which in turn gives a reasonable understanding on statistical and geometric properties of the dynamics. Classical examples include the \emph{Misiurewicz condition} \cite{HM}, \emph{semi-hyperbolicity} \cite{JJ} and the \emph{Collet-Eckmann
condition} ($CE$) \cite{CE,HJ,BC,HN,NS,NP,HP1,HP2,CEH,HCE}. More recent examples include the \emph{Topological Collet-Eckmann condition} ($TCE$) \cite{NP,HCE,ETIC}, summability conditions \cite{HNUH,NUH,ACIP} and the \emph{Collet-Eckmann condition for recurrent critical orbits} ($RCE$) \cite{TRCE,MCN1}.

In the setting of rational dynamics, $CE$ implies $TCE$ \cite{CEH,ETIC}, but there are $TCE$
polynomials which do not verify the $CE$ condition \cite{ETIC}. Graczyk asked
whether $RCE$ is equivalent to $TCE$. In a previous paper, the author
showed that $RCE$ implies $TCE$ \cite{MCN1}. Here, we show that the converse is not
true.

\begin{thma}
There exists a $TCE$ rational map that is not $RCE$.
\end{thma}

On the interval, $TCE$ is equivalent to $CE$ in
the S-unimodal setting \cite{NS,NP}, thus $CE$ is a topological invariant. However, with more than one critical point this
is no longer true. {\'S}wi{\c{a}}tek conjectured that $RCE$ is topologically invariant for multimodal analytic maps with negative Schwarzian derivative \cite{TRCE}. We show this conjecture to be false.

\begin{thmb}
In interval dynamics, the $RCE$ condition for S-multimodal analytic maps is not topologically invariant.
\end{thmb}

\subsection{History}

In one-dimensional dynamics, the orbits of critical points (zeros of the derivative in the smooth case) play a special role. Conditions on critical orbits and their consequences for the dynamics were first studied in the context of interval dynamics.

Two important results in this direction were published in 1981. Misiurewicz showed that S-unimodal maps with non-recurrent critical orbit have an absolutely continuous invariant probability measure \cite{HM}, see \rdefs{Multi}{NSD}. 
Jakobson proved that real quadratic maps which have such invariant measures are abundant, that is, the set of their parameters is of positive measure \cite{HJ}. Collet and Eckmann introduced the $CE$ condition and showed that S-unimodal $CE$ maps have absolutely continuous invariant probability measures \cite{CE}. Benedicks and Carleson showed that $CE$ maps are abundant in the quadratic family \cite{BC}.

Nowicki and Sands showed that $CE$ is equivalent to several non-uniform hyperbolicity conditions for S-unimodal maps \cite{HN,NS}. One of them is \emph{Uniform Hyperbolicity on repelling Periodic orbits} $(UHP)$, see \rdef{UHP}. Shortly after, it was noticed by Nowicki and Przytycki that $CE$ is also equivalent to $TCE$ in this setting \cite{NP}, see \rdef{TCE}. Therefore $CE$ becomes a topological invariant for S-unimodal maps. This is no longer true for S-multimodal maps \cite{NP,TRCE}. {\'S}wi{\c{a}}tek conjectured that $RCE$ (see \rdef{MRCE}) is topologically invariant for multimodal analytic maps with negative Schwarzian derivative, see Conjecture 1 in \cite{TRCE}. We provide a counterexample to this conjecture, Theorem B.

A generalization of the Misiurewicz condition, semi-hyperbolicity, was studied in complex polynomial dynamics by Carleson, Jones and Yoccoz \cite{JJ}, see \rdef{NR}. They show that semi-hyperbolicity is equivalent to $TCE$ with $P=1$ (see \rdef{TCE}), but also to John regularity of Fatou components. They also prove that such polynomials satisfy the \emph{Exponential Shrinking of components condition} $(ExpShrink)$, using a telescopic construction, see \rdef{ES}.

$CE$ rational were initially studied by Przytycki \cite{HP1,HP2}. Later, he showed that $TCE$ implies $CE$ if the Julia set contains only one critical point (unicritical case) \cite{HCE}. Graczyk and Smirnov show that $CE$ implies the \emph{backward or second Collet-Eckmann condition} $(CE2(z_0))$, which in turn implies the H\"older regularity of Fatou components \cite{CEH}, see \rdef{CE2}. The first implication is obtained using a telescopic construction along the backward orbit of $z_0$. Przytycki and Rohde proved that $CE$ implies $TCE$ \cite{CEiTCE}. Therefore, in the unicritical case they are equivalent and $CE$ is topologically invariant. Recently, Aspenberg showed that $CE$ rational maps are abundant \cite{A}.

Przytycki, Rivera-Letelier and Smirnov establish the equivalence of several non-uniform hyperbolicity conditions, as $TCE$, $CE2(z_0)$, $UHP$, $ExpShrink$ and the existence of a positive lower bound for the Lyapunov exponent of invariant measures \cite{ETIC}. A semi-hyperbolic counterexample shows that $TCE$ does not imply $CE$. Another counterexample, involving semi-hyperbolic maps, shows that $CE$ is not a quasi-conformal invariant. 

In an attempt to characterize $TCE$ in terms of properties of critical orbits, the author studied $RCE$ for rational maps (see \rdef{RCE}). In \cite{MCN1} it is shown that $RCE$ implies $TCE$. We provide a counterexample to the converse, Theorem A.

\subsection{A short overview}
In the following section we present some definitions and basic results and lemmas necessary for our study.

In \rsec{Fam},  we describe a technique of building real polynomials with prescribed
topological and analytical properties by specifying their combinatorial properties. 
The critical orbits of the polynomials we shall construct will be on the interval $[0,1]$. Therefore, we can restrict our attention to the dynamics on the unit interval. We shall make use of the theory of \emph{kneading sequences}  to construct our maps.

A kneading sequence is a sequence of symbols associated to the points of a critical orbit. The critical points of an interval map define a partition of the interval. To each element of the partition we associate a symbol. The orbit of a critical point will thus generate a symbol sequence which describes the itinerary of the point through the various partition elements. Knowledge of the kneading sequences is enough to fully describe the combinatorics of a map. Moreover, in the absence of \emph{homtervals} (\rdef{Hom}), two maps with the same kneading sequences are topologically conjugate.

We shall consider one-parameter families of bimodal maps (i.e.\ with two critical points) of the interval, see \rdef{Multi}.  
While the theory of multimodal maps and kneading sequences is generally well understood \cite{MS}, for the most part it is related 
 to topological properties of the dynamics. We develop new tools to obtain a prescribed growth (or lack thereof)
of the derivative on the critical orbits.

In \rsec{UHPnRCE} we prove Theorem A, constructing an $ExpShrink$ polynomial (thus $TCE$) which is not $RCE$. In the vicinity of critical points the diameter of a small domain decreases at most in the power rate, while the derivative
can approach $0$ as fast as one wants (in comparison with the diameter of the domain). This important difference in the behaviour of derivative and diameter is the main idea of the (rather technical) proof. 

In \rsec{RCEnTOP}  we prove Theorem B, a counterexample to the conjecture of {\'S}wi{\c{a}}tek.
Using careful estimates of the derivative on the critical orbits we construct two
polynomials with negative Schwarzian derivative with the same combinatorics, thus topologically 
conjugate on the interval, such that only one is $RCE$. This situation is in sharp contrast with 
the unimodal case, where the Collet-Eckmann condition is topologically invariant \cite{NS, NP}. An important feature
of our counterexample is that the corresponding critical points of this two polynomials are of
different degree. One should be aware that considered as maps of the complex plane they are not conjugate.




\section{Preliminaries}

Let $R$ be a rational map, $J$ its Julia set and $\MO{Crit}$ the set of
critical points.

\begin{definition}
\label{defCE}
We say that $c\in\mathrm{Crit}$ satisfies the
\emph{Collet-Eckmann condition} $(c\in CE)$ if $|(R^n)'(R(c))|>C\lambda^n$ for all
$n>0$ and some constants $C>0,\lambda>1$. We say that $R$ is \emph{Collet-Eckmann}
if all critical points in $J$ are $CE$.
\end{definition}

\begin{definition}
\label{defNR}
Given $c\in\mathrm{Crit}$ we say that it is \emph{non-recurrent} $(c\in NR)$ if
$c\notin \omega(c)$, where $\omega(c)$ is the $\omega$\emph{-limit set},
the set of accumulation points of the orbit $(R^n(c))_{n>0}$. We call $R$
\emph{semi-hyperbolic} if all critical points in $J$ are non-recurrent and $R$
has no parabolic periodic orbits.
\end{definition}

Recurrent Collet-Eckmann condition is weaker than Collet-Eckmann or semi-hyper\-bo\-li\-ci\-ty alone. Indeed, in \cite{CEH} it is shown that a Collet-Eckmann rational map cannot have parabolic cycles.

\begin{definition}
\label{defRCE}
We say that $R$ satisfies the \emph{Recurrent Collet-Eckmann} $(RCE)$ condition
if every recurrent critical point in the Julia set is Collet-Eckmann and $R$
has no parabolic periodic orbits.
\end{definition}

Let us remark that a $RCE$ rational map may have critical points in $J$ that are Collet-Eckmann
and non-recurrent in the same time. Moreover any critical orbit may accumulate on other
critical points.

Several weak hyperbolicity standard conditions are shown to be equivalent in \cite{ETIC}. 
Among these conditions we recall \emph{Topological Collet-Eckmann condition}
$(TCE)$, \emph{Uniform Hyperbolicity on Periodic orbits} $(UHP)$, \emph{Exponential Shrinking of components} $(ExpShrink)$ and 
\emph{Backward Collet-Eckmann condition at some $z_0\in\overline{\mathbb C}$} $(CE2(z_0))$.

Let us define these conditions.

\begin{definition}
\label{defES}
$R$ satisfies the \emph{Exponential Shrinking of components
condition} $(ExpShrink)$ if there are $\lambda>1,\ r>0$ such that
for all $z\in J,\ n>0$ and every connected component $W$ of $R^{-n}(B(z,r))$
$$\MO{diam}{W}<\lambda^{-n}.$$
\end{definition}
Here we use the spherical distance as $\infty$ may be contained in the Julia set.

\begin{definition}
\label{defUHP}
$R$ satisfies \emph{Uniform Hyperbolicity on Periodic orbits} $(UHP)$ if there is $\lambda>1$ such that
for all periodic points $z\in J$ with $R^n(z)=z$ for some $n>0$
$$|(R^n)'(z)|>\la^n.$$
\end{definition}

\begin{definition}
\label{defCE2}
$R$ satisfies the \emph{Backward Collet-Eckmann condition at some $z_0\in\overline{\mathbb C}$} $(CE2(z_0))$ 
if there are $\lambda>1,C>0$ and $z_0\in\RS$ such that for any preimage $z\in\RS$ of $z_0$ with $R^n(z)=z_0$ for some $n>0$
$$|(R^n)'(z)|>C\la^n.$$
\end{definition}

The definition of $TCE$ is technical but it is stated exclusively in topological terms,
therefore it is invariant under topological conjugacy.

\begin{definition}
\label{defTCE}
$R$ satisfies the \emph{Topological Collet-Eckmann condition}
$(TCE)$ if there are $M\geq 0,P\geq 1$ and $r>0$ such that
for all $z\in J$ there exists a strictly increasing sequence of integers $\sq nj1$ such that for all $j\geq 1$,
$n_j\leq P\cdot j$ and
$$\#\SetDef i{0\leq i<n_j,\MO{Comp}_{R^i(z)}R^{-(n_j-i)}B(R^{n_j}(x),r)\cap\cri\neq\es}\leq M,$$
where $\MO{Comp}_y$ means the connected component containing $y$.
\end{definition}

Using the equivalence of these conditions, we may formulate the main result in \cite{MCN1} as follows.
\begin{theorem}
The $RCE$ condition implies $TCE$ for rational maps.
\label{thmMCN1}
\end{theorem}

To produce our counterexamples we restrict to dynamics of real polynomials on the interval with all critical points real.
This is a particular case of multimodal dynamics.

In the sequel all distances and derivatives 
are considered with respect to the Euclidean metric if not specified otherwise.

Let us define multimodal maps and state some classical results about their dynamics. 

\begin{definition}
\label{defMulti}
Let $I$ be the compact interval $[0,1]$ and $f:I\rightarrow I$ a \emph{piecewise strictly monotone}
continuous map. This means that $f$ has a finite number of turning points 
$0<c_1<\ldots<c_l<1$, points where $f$ has a local extremum, and $f$ is strictly monotone
on each of the $l+1$ intervals $I_1=[0,c_1), I_2=(c_1,c_2),\ldots,I_{l+1}=(c_l,1]$. Such a map is called
\emph{$l$-modal} if $f(\partial I)\subseteq\partial I$. If $l=1$ then $f$ is called \emph{unimodal}. 
If $f$ is $C^{1 + r}$ with $r\geq 0$ it is called a smooth $l$-modal map if $f'$ has no zeros outside $\SetEnu{c_1}{c_l}$.
\end{definition}

If $f$ is a $l$-modal map, let us denote by $\MO{Crit}_f$ the set of turning points - or critical points  
$$\MO{Crit}_f=\SetEnu{c_1}{c_l}.$$

Let us define the \emph{Recurrent Collet-Eckmann condition} ($RCE$) in the context of multimodal dynamics. Remark that it is similar to \rdef{RCE}.
\begin{definition}
\label{defMRCE}
We say that $f$ satisfies $RCE$ if every recurrent critical point $c\in\MO{Crit}_f$, $c\in\omega(c)$ is Collet-Eckmann, that is, there exist $C>0,\la>1$ such that for all $n\geq 0$
$$\abs{(f^n)'(f(c))}>C\la^n.$$
\end{definition}
Our counterexamples are polynomials which have all critical points in $I=[0,1]$ which is included in the Julia set, as they do not have attracting or neutral periodic orbits and $I$ is forward invariant. Therefore in this case the previous definition is equivalent to \rdef{RCE}. Analogously, semi-hyperbolicity, $UHP$, $CE2(x)$ and $TCE$ admit very similar definitions to the rational case.

For all $x\in I$ we denote by $O(x)$ or $O^+(x)$ its forward orbit 
$$O(x)=(f^n(x))_{n\geq 0}.$$ Analogously, let $O^-(x)=\SetDef{y\in f^{-n}(x)}{n\geq 0}$ and
$O^\pm(x)=\SetDef{y\in f^{n}(x)}{n\in \mathbb Z}$. We also extend these notations to
orbits of sets. For $S\subseteq I$ let $O^+(S)=\SetDef{f^{n}(x)}{x\in S, n\geq 0}$, 
$O^-(S)=\SetDef{y\in f^{-n}(x)}{x\in S, n\geq 0}$ and $O^\pm(S)=O^+(S)\cup O^-(S)$.

One of the most important questions in all areas of dynamics is when two systems
have similar underlaying dynamics. A natural equivalence relation for multimodal
maps is topological conjugacy.

\begin{definition}
\label{defConj}
We say that two multimodal maps $f,g:I\rightarrow I$ are \emph{topologically conjugate} or
simply \emph{conjugate} if there is a homeomorphism $h:I\rightarrow I$ such that
$$h\circ f=g\circ h.$$
\end{definition}
One may remark that if $f$ and $g$ are conjugate by $h$ then $h(f^n(x))=g^n(h(x))$ for all 
$x\in I$ and $n\geq 0$ so $h$ maps
orbits of $f$ onto orbits of $g$. It is easy to check that $h$ is a monotone bijection form the 
critical set of $f$ to the critical set of $g$. We may also consider combinatorial properties
of orbits and use the order of the points of critical orbits to define another 
equivalence relation between multimodal maps. Theorem II.3.1 in \cite{MS} shows that it is 
enough to consider only the forward orbit of the critical set.
\begin{theorem}
\label{thmEquiv}
Let $f,g$ be two $l$-modal maps with turning points $c_1<\ldots<c_l$ respectively $\tilde c_1<\ldots<\tilde c_l$.
The following properties are equivalent.
\begin{enumerate}
\item There exists an order preserving bijection $h$ from $O^+(\MO{Crit}_f)$ to $O^+(\MO{Crit}_g)$ such that
$$h(f(x))=g(h(x))\mbox{ for all }x\in O^+(\MO{Crit}_f).$$
\item There exists an order preserving bijection $\tilde h$ from $O^\pm(\MO{Crit}_f)$ to $O^\pm(\MO{Crit}_g)$ 
such that
$$\tilde h(f(x))=g(\tilde h(x))\mbox{ for all }x\in O^\pm(\MO{Crit}_f).$$
\end{enumerate}
\end{theorem}
If $f$ and $g$ satisfy the properties of the previous theorem we say that they are \emph{combinatorially equivalent}. Note that if $f$ and $g$ are conjugate by an order preserving homeomorphism $h$ then the 
restriction of $h$ to $O^+(\MO{Crit}_f)$
is an order preserving bijection onto $O^+(\MO{Crit}_g)$ so $f$ and $g$ are combinatorially equivalent.
The converse is true only in the absence of homtervals. It is the case of all the examples in this chapter. 
There is a very convenient way to describe the combinatorial type of a multimodal map using symbolic
dynamics. We associate to every point $x\in I$ a sequence of symbols $\iu(x)$ that we call the 
\emph{itinerary} of $x$. The itineraries $\ku_1,\ldots,\ku_l$ of the critical values $f(c_1),\ldots,f(c_l)$ 
are called the kneading sequences of $f$ and the ordered set of kneading sequences the kneading invariant. 
Combinatorially equivalent multimodal maps have the same kneading invariants but the converse is true 
only in the absence of homtervals.
We use the kneading invariant to describe the dynamics of multimodal maps in one-dimensional families.
We build sequences $(\F_n)_{n\geq 0}$ of compact families of $C^1$ multimodal maps with
$\F_{n + 1} \subseteq \F_n$ for all $n\geq 0$ and obtain our examples as the intersection
of such sequences.

When not specified otherwise, we assume $f$ to be a multimodal map.
\begin{definition}
\label{defAttr}
Let $O(p)$ be a periodic orbit of $f$. This orbit is called \emph{attracting}
if its \emph{basin} 
$$B(p)=\SetDef{x\in I}{f^k(x)\rightarrow O(p)\mbox{ as }k\rightarrow \infty}$$
contains an open set. The \emph{immediate basin} $B_0(p)$ of $O(p)$ is the union of connected
components of $B(p)$ which contain points from $O(p)$. If $B_0(p)$ is a neighborhood of $O(p)$
then this orbit is called a \emph{two-sided attractor} and otherwise a \emph{one-sided attractor}.
Suppose $f$ is $C^1$ and let $m(p)=|(f^n)'(p)|$ where $n$ is the period of $p$. If $m(p)<1$ we say that 
$O(p)$ is attracting respectively super-attracting if $m(p)=0$. We call $O(p)$ \emph{neutral} if $m(p)=1$ 
and we say it is \emph{repelling} if $m(p)>1$.
\end{definition}

Let us denote by $B(f)$ the union of the basins of periodic attracting orbits and by $B_0(f)$ the 
union of immediate basins of periodic attractors. The basins of attracting periodic contain 
intervals on which all iterates of $f$ are monotone. Such intervals do not intersect $O^-(\MO{Crit}_f)$
and they do not carry too much combinatorial information.

\begin{definition}
\label{defHom}
Let us define a \emph{homterval} to be an interval on which $f^n$ is monotone for all $n\geq 0$.
\end{definition}

Homtervals are related to \emph{wandering intervals} and they play an important role in the study of
the relation between conjugacy and combinatorial equivalence.

\begin{definition}
\label{defWand}
An interval $J\subseteq I$ is \emph{wandering} if all its iterates $J,f(J),f^2(J),\ldots$ are disjoint 
and if $(f^n(J))_{n\geq 0}$ does not tend to a periodic orbit.
\end{definition}
  
Homtervals have simple dynamics described by the following lemma, Lemma II.3.1 in \cite{MS}.
\begin{lemma}
\label{lemHom}
Let $J$ be a homterval of $f$. Then there are two possibilities:
\begin{enumerate}
\item $J$ is a wandering interval;
\item $J\subseteq B(f)$ and some iterate of $J$ is mapped into an interval $L$ such that
$f^p$ maps $L$ monotonically into itself for some $p\geq 0$.
\end{enumerate}
\end{lemma}
  
Multimodal maps satisfying some regularity conditions have no wandering intervals. Let us say that 
$f$ is \emph{non-flat} at a critical point $c$ if there exists a $C^2$ diffeomorphism 
$\phi:\mathbb R\rightarrow I$ with $\phi(0)=c$ such that $f\circ\phi$ is a polynomial near
the origin.

The following theorem is Theorem II.6.2 in \cite{MS}.
\begin{theorem}
\label{thmWand}
Let $f$ be a $C^2$ map that is non-flat at each critical point. Then $f$ has
no wandering intervals.
\end{theorem}
Guckenheimer proved this theorem in 1979 for unimodal maps with \emph{negative Schwarzian derivative}
with \emph{non-degenerate} critical point, that is with $|f''(c)|\neq 0$. The Schwarzian derivative was
first used by Singer to study the dynamics of quadratic unimodal maps $x\rightarrow ax(1-x)$ with 
$a\in [0,4]$. He observed that this property is preserved under iteration and that is has important
consequences in unimodal and multimodal dynamics.
\begin{definition}
\label{defNSD}
Let $f:I\rightarrow I$ be a $C^3$ $l$-modal map. The \emph{Schwarzian derivative} of $f$ at $x$ is
defined as
$$Sf(x)=\frac{f'''(x)}{f'(x)}-\frac 32\left(\frac{f''(x)}{f(x)}\right)^2,$$
for all $x\in I\setminus\SetEnu{c_1}{c_l}$.
\end{definition}

We may compute the Schwarzian derivate of a composition 
\begin{equation}
\label{equSchw}
S(g\circ f)(x)=Sg(f(x))\cdot|f'(x)|^2+Sf(x),
\end{equation}
therefore if $Sf<0$ and $Sg<0$ then $S(f\circ g)<0$ so negative Schwarzian derivative is preserved under
iteration. Let us state an important consequence of this property for $C^3$ maps of the interval proved
by Singer (see Theorem II.6.1 in \cite{MS}).
\begin{theorem}[Singer]
\label{thmSing}
If $f:I\rightarrow I$ is a $C^3$ map with negative Schwarzian derivative then
\begin{enumerate}
\item the immediate basin of any attracting periodic orbit contains either a critical point of $f$
or a boundary point of the interval $I$;
\item each neutral periodic point is attracting;
\item there are no intervals of periodic points.
\end{enumerate}
\end{theorem}

Combining this result with Theorem \ref{thmWand} and Lemma \ref{lemHom} we obtain the following
\begin{corollary}
\label{corHom}
If $f$ is $C^3$ multimodal map with negative Schwarzian derivative that is non-flat
at each critical point and which has no attracting periodic orbits then it has no 
homterval. Therefore $O^-(\MO{Crit}_f)$ is dense in $I$.
\end{corollary}

The following corollary is a particular case of the corollary of Theorem 
II.3.1 in \cite{MS}.
\begin{corollary}
\label{corEquiv}
Let $f,g$ and $h$ be as in Theorem \ref{thmEquiv}. If $f$ and $g$ have no homtervals then they are 
topologically conjugate.
\end{corollary}

All our examples of multimodal maps in this chapter are polynomials with negative Schwarzian 
derivative and without attracting periodic orbits. We prefer however to use slightly more general 
classes of multimodal maps, as suggested by the previous two corollaries. As combinatorially equivalent
multimodal maps have the same monotonicity type we only use maps that are increasing on the leftmost lap
$I_1$, that is exactly the multimodal maps $f$ with $f(0)=0$. Let us define some classes of multimodal maps
$$\mathcal S_l=\SetDef{f:I\rightarrow I}{f\mbox{ is a }C^1 l\mbox{-modal map with }f(0)=0},$$
$$\mathcal S'_l=\SetDef{f\in\mathcal S_l}{f\mbox{ is }C^3\mbox{ and }Sf<0},$$
$$\Pc_l=\SetDef{f\in \mathcal S'_l}{f\mbox{ non-flat at each critical point}}\mbox{ and}$$
\label{defP2}
$$\Pc'_l=\SetDef{f\in \Pc_l}{\mbox{all periodic points of }f\mbox{ are repelling}}.$$

We have seen that in the absence of homtervals combinatorially equivalent multimodal maps
are topologically conjugate. Using symbolic dynamics is a more convenient way to describe
the combinatorial properties of forward critical orbits. Let $\A_I=\SetEnu{I_1}{I_{l+1}}$ and
$\A_c=\SetEnu{c_1}{c_l}$ be two alphabets and $\A=\A_I\cup\A_c$. Let 
$$\Sigma=\A_I^{\mathbb N}\cup\bigcup_{n\geq 0}\left(\A_I^n\times \A_c\right)$$
be the space of sequences of symbols of $\A$ with the following 
property. If $\iu\in\Sigma$ and $m=|\iu|\in\overline{\mathbb N}$ is its length then $m=\infty$ if and
only if $\iu$ consists only of symbols of $\A_I$. Moreover, if $m<\infty$ then $\iu$ contains
exactly one symbol of $\A_c$ on the rightmost position. Let $\Sigma'=\Sigma\setminus\A_c$ 
be the space of sequences $\iu\in\Sigma$ with $|\iu|>1$. Let us define the shift transformation
$\sigma:\Sigma'\rightarrow\Sigma$ by
$$\sigma(i_0i_1\ldots)=i_1i_2\ldots.$$
If $f\in \Sc_l$ let $\iu:I\rightarrow \Sigma$ be defined by $\iu(x)=i_0(x)i_1(x)\ldots$
where $i_n(x)=I_k$ if $f^n(x)\in I_k$ and $i_n(x)=c_k$ if $f^n(x)=c_k$ for all $n\geq 0$. The map
$\iu$ relates the dynamics of $f$ on $I\setminus\SetEnu{c_1}{c_l}$ with the shift transformation
$\sigma$ on $\Sigma'$
$$\iu(f(x))=\sigma(\iu(x))\mbox{ for all }x\in I\setminus\SetEnu{c_1}{c_l}.$$
Moreover, we may define a \emph{signed lexicographic ordering} on $\Sigma$ that makes
$\iu$ increasing. It becomes strictly increasing in the absence of homtervals.

\begin{definition}
\label{defOrder}
A \emph{signed lexicographic ordering} $\prec$ on $\Sigma$ is defined as follows. Let us
define a sign $\epsilon:\A\rightarrow\{-1,0,1\}$ where $\epsilon(I_j)=(-1)^{j+1}$ for 
all $j\var 1{l+1}$ and $\epsilon(c_j)=0$ for all $j\var 1{l}$. Using the natural 
ordering on $\A$ we say that $\underline x\prec\underline y$ if there exists $n\geq 0$
such that $x_i=y_i$ for all $i\var 0{n-1}$ and
$$x_n\cdot\prod_{i=0}^{n-1}\epsilon(x_i)<y_n\cdot\prod_{i=0}^{n-1}\epsilon(y_i).$$
\end{definition}
Let us observe that $\prec$ is a complete ordering and that $\epsilon\cdot f'>0$ on $\rset$,
that is $\epsilon $ represents the monotonicity of $f$. The product $\prod_{i=0}^{n-1}\epsilon(x_i)$
represents therefore the monotonicity of $f^n$. This is the main reason for the monotonicity of $\iu$ 
with respect to $\prec$.

\begin{prop}
\label{propOrder}
Let $f\in\Sc_l$ for some $l\geq 0$.
\begin{enumerate}
\item If $x<y$ then $\iu(x)\preceq\iu(y)$.
\item If $\iu(x)\prec\iu(y)$ then $x<y$.
\item If $f\in\Pc'_l$ then $x<y$ if and only if $\iu(x)\prec\iu(y)$.
\end{enumerate}
\end{prop}
\begin{proof}
The first two points are Lemma II.3.1 in \cite{MS}. If $f\in\Pc'_l$ then by Corollary 
\ref{corHom} $O^-(\MO{Crit}_f)$ is dense in $I$. Let us note that 
$$O^-(\MO{Crit}_f)=\SetDef{x\in I}{|\iu(x)|<\infty}.$$
Moreover, $O^-(\MO{Crit}_f)$ is countable as $f^{-1}(x)$ is finite for all $x\in I$,
therefore $\iu$ is strictly increasing.
\end{proof}

Let us define the kneading sequences of $f\in\Sc_l$ by $\ku_i=\iu(f(c_i))$ for 
$i\var 1l$, the itineraries of the critical values. The kneading invariant of $f$ is
$\Ku(f)=(\ku_1,\ldots,\ku_l)$. The last point of the previous lemma shows that if
$f,g\in\Pc'_l$ and $\Ku(f)=\Ku(g)$ then there is an order 
preserving bijection $h:O^+(\MO{Crit}_f)\rightarrow O^+(\MO{Crit}_g)$. Therefore, by Corollaries 
\ref{corHom} and \ref{corEquiv}, $f$ and $g$ are topologically conjugate. 

Let us define one-dimensional smooth families of multimodal maps. They are the central object of this paper.

\begin{definition}
\label{defFam}
We say that $\F:[\al,\be]\rightarrow \Sc_l$ is a \emph{family of $l$-modal maps}
if $\F$ is continuous with respect to the $C^1$ topology of $\Sc_l$.
\end{definition}

Note that we do not assume the continuity of critical points in such a family - as in the
general definition of a family of multimodal maps in \cite{MS} - as it is a direct consequence
of the smoothness conditions we impose.

When not stated otherwise we suppose $\F:[\al,\be]\rightarrow\Sc_l$ is a family
of $l$-modal maps and denote $f_\ga=\F(\ga)$.
\begin{lemma}
\label{lemCrit}
The critical points $c_i:[\al,\be]\rightarrow I$ of $f_\ga$ are continuous maps for 
all $i\var 1l$.
\end{lemma}
\begin{proof}
Fix $\ga_0\in[\al,\be]$ and $$0<\ve<\frac 12\min_{i\neq j}|c_i(\ga_0)-c_j(\ga_0)|.$$
Let $A=\SetDef{x\in[0,1]}{\ve\leq\min_i|x-c_i(\ga_0)|}$,
a finite union of compact intervals and
$$\theta=\min_{x\in A}|f_{\ga_0}'(x)|>0$$
by Definition \ref{defMulti}. Then the monotonicity of $f_{\ga_0}$ alternates on the connected components of $A$. Let $\delta>0$ be such that
$||f_\ga-f_{\ga_0}||_{C^1}<\frac\theta 2$ for all 
$\ga\in(\ga_0-\delta,\ga_0+\delta)\cap[\al,\be].$
Therefore the critical points $c_i(\ga)$ satisfy
$$|c_i(\ga)-c_i(\ga_0)|<\ve$$
for all $i\var 1l$ and $\ga\in(\ga_0-\delta,\ga_0+\delta)\cap[\al,\be]$
as $f'_\ga(x)\cdot f'_{\ga_0}(x)>0$ for all $x\in A$.
\end{proof}

Let us show that the $C^1$ continuity of families of multimodal maps is preserved under iteration.
\begin{lemma}
\label{lemIter}
Let $G,H\fd{[a,b]}{C^1(I,I)}$ be continuous. Then the map 
$$c\ra G(c)\circ H(c)\mbox{ is continuous on }[a,b].$$
\end{lemma}
\def\cin{{c\in(c_0-\de,c_0+\de)\cap[a,b]}}
\def\acm{{\mbox{ for all }\cin}}
\def\ac{{ for all $\cin$}}
\def\gc{g_{c_0}}
\def\hc{h_{c_0}}
\begin{proof}
Fix $c_0\in[a,b]$ and $\ve>0$. We show that there is $\de>0$ such that 
$$\nc{G(c)\circ H(c)-G(c_0)\circ H(c_0)}<\ve\acm.$$
For transparency we denote $g_c=G(c)$ and $h_c=H(c)$ for all $c\in[a,b]$. Let
$$M=\max\SetDef{\nc{g_c},\nc{h_c},1}{c\in[a,b]}.$$
As $\gc'$ is uniformly continuous on $I$, there is $\de'>0$
such that 
$$|\gc'(x)-\gc'(y)|<\frac\ve{4M}\mbox{ for all }x,y\in I\mbox{ with }|x-y|<\de'.$$
Let $\de>0$ such that
$$\sup\SetDef{\nc{g_c-\gc},\nc{h_c-\hc}}{\cin}<\min\left(\frac\ve{4M},\de'\right).$$
We compute a bound for $\nc{g_c\circ h_c-\gc\circ\hc}$\ac
$$
\begin{array}{rcl}
 ||g_c\circ h_c-\gc\circ\hc||_\ft & \leq & ||g_c\circ h_c - g_c\circ\hc||_\ft + ||g_c\circ\hc - \gc\circ\hc||_\ft\\
                            & \leq & M\frac\ve{4M} + \frac\ve{4M}\\
                            & < & \ve.
\end{array}
$$
Analogously
$$
\begin{array}{rcl}
 ||g_c'\circ h_c\cdot h_c'-\gc'\circ\hc\cdot \hc'||_\ft & \leq & 
                 ||(g_c'\circ h_c - \gc'\circ h_c)\cdot h_c'||_\ft + \\
             & & ||(\gc'\circ h_c - \gc'\circ\hc)\cdot h_c'||_\ft + \\
             & & ||\gc'\circ\hc\cdot (h_c' - \hc')||_\ft\\
                            & \leq & \frac\ve{4M}M + \frac\ve{4M}M + M\frac\ve{4M}\\
                            & < & \ve
\end{array}
$$
as $||h_c - \hc||_\ft<\de'$.
\end{proof}
Remark that by iteration $\ga\ra f_\ga^n$ is continuous for all $n\geq 1$. 

The following proposition shows that pullbacks of given combinatorial type of continuous maps
are continuous in a family of multimodal maps.
\begin{prop}
\label{propPull}
Let $y:[\al,\be]\rightarrow \stackrel{\circ}{I}$ be continuous and $S\in\A^n_I$, where $\stackrel{\circ}{I}$ denotes the interior of $I$. 
A maximal connected domain of definition $D$ of the map $\ga\rightarrow x_\ga$ such that
$$
\begin{array}{l}
f^n_\ga(x_\ga)=y(\ga)\mbox{ and}\\
\iu(x_\ga)\in S\ts
\end{array}
$$
is open in $[\al,\be]$ and $\ga\rightarrow x_\ga$ is unique and continuous on $D$.
\end{prop}
\begin{proof}
Suppose that for some $\ga$ there are $x_1<x_2\in \stackrel{\circ}{I}$ with $f^n_\ga(x_1)=f^n_\ga(x_2)=y(\ga)$ 
and such that $\iu(x_1)=\iu(x_2)=S\iu(y(\ga))$ for some $\ga\in[\al,\be]$. But $S\in\A_I^n$  
so $f^n$ is strictly monotone on $[x_1,x_2]$, which contradicts $f^n_\ga(x_1)=f^n_\ga(x_2)$ so
$\ga\rightarrow x_\ga$ is unique.

Let $x_{\ga_0}$ be as in the hypothesis and $\ve >0$ such that $(x_{\ga_0}-\ve,x_{\ga_0}+\ve)\se\stackrel{\circ}{I}$. We show that there exists
$\delta>0$ such that $\ga\rightarrow x_\ga$ is defined on $(\ga_0-\delta,\ga_0+\delta)\cap[\al,\be]$
and takes values in $(x_{\ga_0}-\ve, x_{\ga_0}+\ve)$. Let
$$\theta=(f_{\ga_0}^n)'(x_{\ga_0})\neq 0$$
and by eventually diminishing $\ve$ we may suppose that 
$$|(f_{\ga_0}^n)'(x)- \theta|<\frac\theta 4\mbox{ for all }x\in
(x_{\ga_0}-\ve,x_{\ga_0}+\ve).$$
Let $\delta_1>0$ be such that 
$$||f_\ga^n-f_{\ga_0}^n||_{C^1}<\frac{\theta\ve}4<\frac{\theta}4\mbox{ for all }\ga\in
(\ga_0-\delta_1,\ga_0+\delta_1)\cap[\al,\be].$$
Let also $\delta_2>0$ be such that 
$$|y(\ga)-y(\ga_0)|<\frac{\theta\ve}4\mbox{ for all }
\ga\in(\ga_0-\delta_2,\ga_0+\delta_2)\cap[\al,\be].$$
We choose $\delta=\min(\delta_1,\delta_2)$ and show that
$$y(\ga)\in f^n_\ga((x_{\ga_0}-\ve, x_{\ga_0}+\ve))\mbox{ for all }
\ga\in(\ga_0-\delta,\ga_0+\delta)\cap[\al,\be].$$
Indeed, $f^n_\ga$ is monotone on $(x_{\ga_0}-\ve, x_{\ga_0}+\ve)$ and
$$|f^n_\ga(x_{\ga_0}\pm\ve)-y(\ga_0)|>\frac{\theta\ve}4$$
for all $\ga\in(\ga_0-\delta,\ga_0+\delta)\cap[\al,\be]$ as
$|f^n_\ga(x_{\ga_0}\pm\ve)-y(\ga_0)|=|f^n_\ga(x_{\ga_0}\pm\ve)-
f^n_{\ga_0}(x_{\ga_0}\pm\ve)+f^n_{\ga_0}(x_{\ga_0}\pm\ve)-
f^n_{\ga_0}(x_{\ga_0})|$ and $|f^n_{\ga_0}(x_{\ga_0}\pm\ve)-
f^n_{\ga_0}(x_{\ga_0})|>\frac 34\theta\ve$.
\end{proof}

As an immediate consequence of the previous proposition and Lemma \ref{lemCrit} we obtain the 
following corollary.
\begin{corollary}
\label{corCont}
If $\F$ realizes a finite itinerary sequence $\iu_0\in\Sigma$, that is for all 
$\ga\in[\al,\be]$ there is $x(\iu_0)(\ga)\in I$ such that 
$$\iu(x(\iu_0)(\ga))=\iu_0,$$
then $x(\iu_0):[\al,\be]\rightarrow I$ is unique and continuous.
\end{corollary}

One may observe that if $x,y:[\al,\be]\rightarrow I$ are continuous and
for some $k\geq 0$ 
$$\left(f_\al^k(x(\al))- y(\al)\right)\cdot\left(f_\be^k(x(\be))- y(\be)\right)<0$$
then there exists $\ga\in[\al,\be]$ such that 
\begin{equation}
\label{equSol}
f_\ga^k(x(\ga)) = y(\ga).
\end{equation}
Therefore if $\iu(x(\al))\neq\iu(x(\be))$ then there exists
$\ga\in[\al,\be]$ such that $\iu(x(\ga))$ is finite. Let 
$m=\min\SetDef{k\geq 0}{\exists\ga\in[\al,\be]\mbox{ such that }\iu(x(\al))(k)\neq\iu(x(\ga))(k)}$ then the itinerary
$\sigma^m\iu(x(\ga))=\iu(f_\ga^m(x(\ga)))$ changes the first symbol on $[\al,\be]$.
Without loss of generality we may assume that $\sigma^m\iu(x(\al))\prec\sigma^m\iu(x(\be))$. 
Therefore there exists $i\in\SetEnu 1l$ such that $f_\ga^m(x(\al))\leq c_i(\al)$ and 
$f_\ga^m(x(\be))\geq c_i(\al)$, which yields $\ga$ using the previous remark. 

A simplified version of the proof of Proposition \ref{propPull} shows that if 
$F:[\al,\be]\rightarrow C^1(I)$ is continuous, $r_0\in I$ is a root of $F(\ga_0)$
and $(F(\ga_0))'(r_0)\neq 0$ then there are $J\subseteq [\al,\be]$ a neighborhood of $\ga_0$ 
and $r:J\rightarrow I$ continuous such that $F(\ga)(r(\ga))=0$ for all $\ga\in J$.
For $F(\ga)(x)=f_\ga^n(x)-x$ we obtain the following corollary.
\begin{corollary}
\label{corRep}
Let $r_0$ be a periodic point of $f_{\ga_0}$ of period $n\geq 1$ that is not neutral. There
exists a connected neighborhood $J\subseteq [\al,\be]$ of $\ga_0$ and $r:J\rightarrow I$ 
continuous such that $r(\ga)$ is a non-neutral periodic point of $f_\ga$ of period $n$. 
Moreover, provided $r(\ga)$ is not super-attracting for any $\ga\in J$, the itinerary $\iu(r(\ga))$ 
is constant.
\end{corollary}
\begin{proof}
As a periodic point, $r(\ga)$ exists and is continuous on a connected neighborhood 
$J_0$ of $\ga_0$, using the previous remark. As $|(f_{\ga_0}^n)'(r_0)| \neq 1$,
there is a connected neighborhood $J_1$ of $\ga_0$ such that
$$|(f_{\ga}^n)'(r(\ga))| \neq 1\mbox{ for all }\ga\in J_1.$$
Let $J=J_0\cap J_1$ so $r(\ga)$ is a non-neutral periodic point of period $n$ for all $\ga\in J$.
Suppose that its itinerary $\iu(r(\ga))$ is not constant, then there is $\ga_1\in J$ such
that $\iu(r(\ga_1))$ is finite so the orbit of $r(\ga_1)$ contains a critical point thus it is 
super-attracting.
\end{proof}

Let us define the \emph{asymptotic kneading sequences} $\ku^-_j(\ga)$ and $\ku^+_j(\ga)$ \allg{} and $j\var 1l$.
When they exist, the asymptotic kneading sequences capture important information about the local variation
of the kneading sequences.
\begin{definition}
\label{defAsy}
Let $j\in\SetEnu 1l$ and $\ga\in[\al,\be]$.
If $\ga>\al$ and for all $n\geq 0$ there exists $\de>0$ such that $\ku_j(\ga-\te)\in S_n\times\Sigma$ with
$S_n\in\A_I^n$ for all $\te\in(0,\de)$ then we set $\ku^-_j(\ga)(k)=S_n(k)$ for all $0\leq k<n$. Analogously,
if $\ga<\be$ and for all $n\geq 0$ there exists $\de>0$ such that $\ku_j(\ga+\te)\in S'_n\times\Sigma$ with
$S'_n\in\A_I^n$ for all $\te\in(0,\de)$ then we set $\ku^+_j(\ga)(k)=S'_n(k)$ for all $0\leq k<n$.
\end{definition}

Let us define a sufficient condition for the existence of the asymptotic kneading sequences
\allg{}.
\begin{definition}
\label{defNat}
We call a family $\F\fd{[\al,\be]}{\Sc_l}$ of $l$-modal maps \emph{natural} if for all $j=1,\ldots, l$ the set
$$\ku_j^{-1}(\iu)=\SetDef{\ga\in[\al,\be]}{\ku_j(\ga)=\iu}\mbox{ is finite for all }\iu\in\Sigma
\mbox{ finite.}$$
\end{definition}
This property does not hold in general for $C^1$ families of multimodal maps, even polynomial,
as such a family could be reparametrized to have intervals of constancy in the parameter space.
It is however generally true for analytic families such as the quadratic family $a\ra ax(1-x)$ with
$a\in[0,4]$.

The following proposition shows that this property guarantees the existence of all asymptotic kneading sequences.
\begin{prop}
\label{propAsy}
Let $\F\fd{[\al,\be]}{\Sc_l}$ be a natural family of $l$-modal maps and $j\in\SetEnu 1l$. Then $\ku^-_j(\ga)$ 
exists for all $\ga\in(\al,\be]$ and $\ku^+_j(\ga)$ exists for all $\ga\in[\al,\be)$. Moreover, if 
$\ku_j(\ga)\in\A_I^\ft$ for some $\ga\in(\al,\be)$ then $\ku^-_j(\ga)=\ku_j(\ga)=\ku^+_j(\ga)$. 
If $\ku_j(\ga)=Sc_i$ with $S\in\A_I^n$ for some $n\geq 0$ and $i\in\SetEnu 1l$ then 
$\ku^-_j(\ga)=Sl_1l_2\ldots$ and $\ku^+_j(\ga)=Sr_1r_2\ldots$ with $l_1,r_1\in\left\{I_i,I_{i+1}\right\}$.
\end{prop}
\begin{proof}
If $\F$ is natural then the set of all $\ga\in[\al,\be]$ that have at least one kneading sequence
of length at most $n$ for some $n>0$
$$K_n=\bigcup_{j=1}^l\SetDef{\ga\in[\al,\be]}{|\ku_j(\ga)|\leq n}$$
is finite. This is sufficient for the existence of all asymptotic kneading sequences.

If $\ku_j(\ga_0)\in S\times\Sigma$ with $S\in\A_I^n$ and $n\geq 0,j\in\SetEnu 1l$ then by the continuity of 
$\ga\ra f_\ga^m(c_j)$ and of $\ga\ra c_i$ for all $m\var 0{n-1}$ and $i\var 1l$ there exists $\de>0$ such that 
$$\ku_j(\ga)\in S\times\Sigma\mbox{ for all }\ga\in(\ga_0-\de,\ga_0+\de)\cap[\al,\be].$$
Therefore if $\ku_j(\ga)\in\A_I^\ft$ then $\ku^-_j(\ga)=\ku_j(\ga)=\ku^+_j(\ga)$. If $\ku_j(\ga)=Sc_i$ for some 
$i\in\SetEnu 1l$ then $\ku^-_j(\ga)=Sl_1l_2\ldots$ and $\ku^+_j(\ga)=Sr_1r_2\ldots$. Again by the continuity
of $\ga\ra f_\ga^n(c_j)$ and of $\ga\ra c_k$ for all $k\var 1l$ 
$$l_1,r_1\in\left\{I_i,I_{i+1}\right\}.$$
\end{proof}
Note that we may omit the parameter $\ga$ 
whenever there is no danger of confusion but $c_j$, $\iu$ and $\ku_j$ for some $j\in\SetEnu 1l$ should always
be understood in the context of some $f_\ga$. However, the symbols of the itineraries 
of $\Sigma$ are $I_1,\ldots,I_{l+1},c_1,\ldots,c_l$ and do not depend on $\ga$.


\section{One-parameter families of bimodal maps}
\label{sectFam}

In this section we consider a natural family $\G:[\al,\be]\rightarrow \Pc_2$ 
of bimodal polynomials with negative Schwarzian derivative satisfying the following conditions  
 
\begin{equation}
\label{equGRep}
0,1\in\partial I\mbox{ are fixed and repelling for }g_\al,
\end{equation}
\begin{equation}
\label{equGPre}
g_\ga(c_1)=1\ag{},
\end{equation}
\begin{equation}
\label{equGAl}
g_\ga(c_2)=0\mbox{ if an only if }\ga=\al.
\end{equation} 

Let us denote by $v_n=g_\ga^{n+1}(c_2)$ for $n\geq 0$ 
the points of the second critical orbit and let $\ku=\ku_2(\ga)=k_0k_1\ldots$.
If $S\in\A^k_I,k\geq 1$ and $n\geq 1$ we write $S^n$ for $SS\ldots S\in\A^{kn}_I$ 
repeated $n$ times and $S^\infty$ for $SS\ldots\in\A^\infty_I$. 

Proposition \ref{propAsy} shows the existence of $\ku^+(\al)=\ku(\al)=I_1^\ft$ therefore,
there is $\de_0>0$ such that 
\begin{equation}
\label{equGKn}
\ku\in I_1^2\times\Sigma
\end{equation}
for all $\ga\in[\al,\al+\de_0]$. Figure \ref{figGGa} represents the graph
of a bimodal map with the second kneading sequence $I_1c_1\succ\ku(\ga)$ for all $\ga\in[\al,\al+\de_0]$.

\begin{figure}
\begin{center}

\includegraphics[width=2.5in]{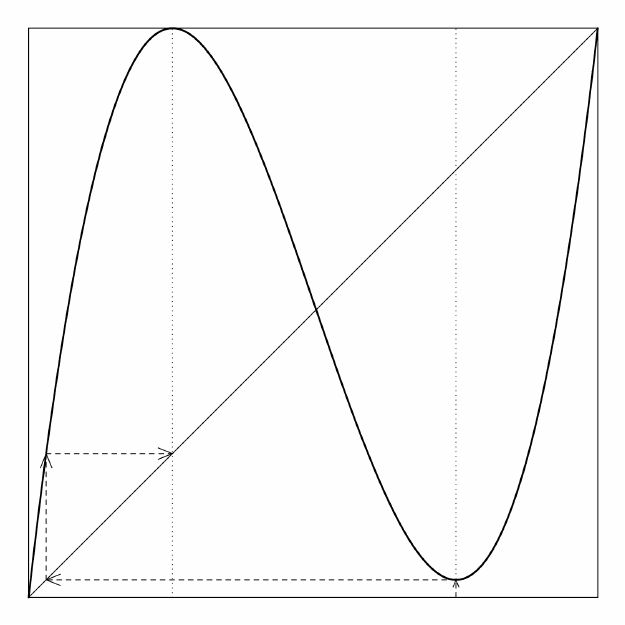}
  
\begin{picture}(0,0)
\put(-44, 12) {$c_1$}
\put(38, 12) {$c_2$}
\put(-80, 12) {$v$}
\end{picture}

\caption{bimodal map with $\ku_2=I_1c_1$.}
\label{figGGa}

\end{center}
\end{figure}

Let us observe that $O^+(\MO{Crit}_{g_\al})=\{0,c_1,c_2,1\}$ and that by Singer's Theorem \ref{thmSing},
$g_\al$ has no homtervals. Therefore by Corollary \ref{corEquiv}, if $\Hc\fd{[\al',\be']}{\Pc_2}$ is a
natural family satisfying \cG then $g_\al$ and $h_{\al'}$ are topologically conjugate. Moreover, $g_\al$ 
is conjugate to the second Chebyshev polynomial (on $[-2,2]$) and topological properties of its dynamics 
are universal. Let us study this dynamics and extend by continuity some of its properties to some 
neighborhood of $\al$ in the parameter space. 

We have seen that $g_\al$ has no homtervals and that all its periodic points are repelling. Proposition
\ref{propOrder} shows that the map
$$\iu(g_\al)\fd I\Sigma\mbox{ is strictly increasing.}$$ 
Let us denote by $\sigma^-(\iu)$ the set of all preimages of $\iu$ by some shift 
$$\sigma^-(\iu)=\SetDef{\iu'\in\Sigma}{\exists k\geq 0\mbox{ such that }\sigma^k(\iu')=\iu}.$$ 
As $(0,1)=\intI\subseteq g_\al(I_j)$ for $j=1,2,3$, $g_\al(c_1)=1$, $g_\al(c_2)=0$, $\iu(g_\al)(0)=I_1^\ft$ and $\iu(g_\al)(1)=I_3^\ft$
$$\iu(g_\al)(\intI)=\Sigma\sm\left(\sigma^-(I_1^\ft)\cup\sigma^-(I_3^\ft)\right).$$
Let us denote by $\Sigma_0=\iu(g_\al)(I)=\iu(g_\al)(\intI)\cup\{I_1^\ft,I_3^\ft\}$. Then
\begin{equation}
\iu(g_\al)\fd I\Sigma_0\mbox{ is an order preserving bijection}.
\label{equSZ}
\end{equation}
Remark also that $\Sigma_0$ is the space of all itinerary sequences of $I$ under a bimodal map. 

As $g_\al$ is decreasing on $I_2$, $g_\al(c_1)>c_1$ and $g_\al(c_2)<c_2$ it has exactly one fixed point 
$r\in I_2$ and it is repelling. Moreover, $g_\al$ has no fixed points in $I_1$ or $I_3$ other than $0$ and $1$
as this would contradict the injectivity of $\iu(g_\al)$. As $0$ and $1$ are repelling fixed points
$g_\al(x)>x$ for all $x\in(0,c_1)$ and $g_\al(x)<x$ for all $x\in(c_2,1)$. Then by the $C^1$ continuity
of $\G$ and Corollary \ref{corRep} we obtain the following lemma.
\begin{lemma}
\label{lemFix}
There is $\de_1>0$ such that $g_\ga$ has exactly one fixed point $r(\ga)$ in $(0,1)$ and all its 
fixed points $0,1$ and $r(\ga)$ are repelling for all
$\ga\in[\al,\al+\de_1]$. Moreover, the map $\ga\ra r(\ga)$ is continuous and $\iu(r)=I_2^\ft$.
\end{lemma}

Let $p$ be a periodic point of period $2$ of $g_\al$. Then $\iu(p)$ is periodic of period $2$ and
infinite. So $\iu(p)\in\SetDef{(I_jI_k)^\ft}{j,k=1,2,3}$.
But $\iu(g_\al)$ is injective, $\iu(g_\al)(0)=I_1^\ft$, $\iu(g_\al)(r)=I_2^\ft$ and 
$\iu(g_\al)(1)=I_3^\ft$ so 
$$\iu(p)\in\SetDef{(I_jI_k)^\ft}{j\neq k\mbox{ and }j,k=1,2,3}\se\Sigma_0.$$
Therefore $g_\al$ has exactly $3$ periodic orbits of period $2$ with itinerary sequences
$(I_1I_2)^\ft$, $(I_1I_3)^\ft$, $(I_2I_3)^\ft$ and their shifts. 
Figure \ref{figG02} illustrates the periodic orbits of period $2$ of $g_\al$. By continuity of
$\ga\ra g_\ga^2$ and Corollary \ref{corRep} we obtain the following lemma.
\begin{lemma}
\label{lemPer}
There is $\de_2>0$ such that $g_\ga$ has exactly $3$ periodic orbits of period $2$
with itinerary sequences
$(I_1I_2)^\ft$, $(I_1I_3)^\ft$, $(I_2I_3)^\ft$ for all $\ga\in[\al,\al+\de_2]$. 
Moreover, the periodic orbits of period $2$ are repelling and continuous with respect to 
$\ga$ on $[\al,\al+\de_2]$.
\end{lemma}

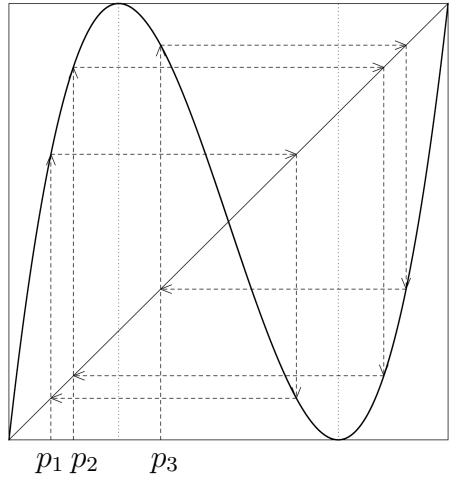
\begin{figure}
\begin{center}

\includegraphics[width=2.5in]{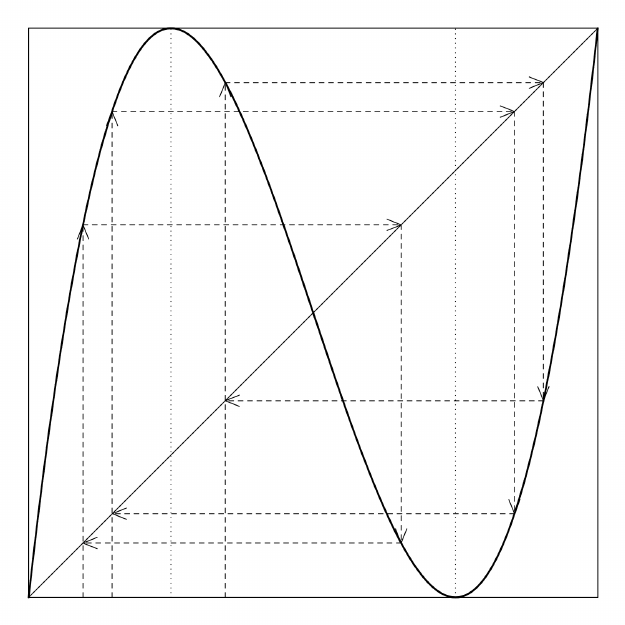}

\begin{picture}(0,0)
\put(-72, 12) {$p_1$}
\put(-59, 12) {$p_2$}
\put(-29, 12) {$p_3$}
\end{picture}

\caption{$g_\al$ and its periodic orbits of period $2$, $p_1$ with $\iu(p_1)=(I_1I_2)^\infty$,
$p_2$ with $\iu(p_2)=(I_1I_3)^\infty$ and $p_3$ with $\iu(p_3)=(I_2I_3)^\infty$.}
\label{figG02}

\end{center}
\end{figure}

Let us define 
\begin{equation}
\label{equBp}
\be'=\al+\min\{\de_0,\de_1,\de_2\}
\end{equation}
so that $\G$ satisfies equality (\ref{equGKn}), Lemma \ref{lemFix} and the previous lemma for all 
$\ga\in[\al,\be']$.

Let us consider the dynamics of all maps $g_\ga$ with $\ga\in[\al,\be']$ from the combinatorial
point of view. We observe that if $x\geq v=g_\ga(c_2)$ then $g_\ga^n(x)\geq v$ for all $n\geq 0$.
This means that any itinerary of $g_\ga$ is of the form $\iu_\ga=I_1^ka\ldots\in\Sigma_0$ with 
$k\geq 0$, $a\neq I_1$ and such that $\sigma^{k+p}\iu_\ga\succeq \ku$ for all $p\geq 0$. Let 
$\Sigma(\ku)$ denote the set of itineraries satisfying this condition. We observe that 
$(v,1)\se g_\ga(I_j)$ for $j=1,2,3$ and $c_1,c_2\in (v,1)$ for all $\ga\in [\al,\be']$ 
by relation (\ref{equGKn}) so we obtain the following lemma. The continuity is an immediate consequence
of Proposition \ref{propPull}.
\begin{lemma}
\label{lemFin}
Let $\ga_0\in [\al,\be']$ and $\ku=\ku_2(\ga_0)$. Then every finite itinerary
$$\iu_0\in\SetDef{\iu\in\Sigma(\ku)}{|\iu|<\ft}$$
is realized by a unique point $x(\iu)\in I$ and $\ga\ra x(\iu)$ is continuous on a neighborhood of $\ga_0$.
\end{lemma}

A kneading sequence $\ku\in \Sigma(\ku)$ satisfies the following property.
\begin{definition}
\label{defMin}
We call $\um\in\Sigma_0$ \emph{minimal} if
$$\um\preceq\sigma^k\um\mbox{ for all }0\leq k<|m|.$$
\end{definition}

The following proposition shows that the minimality is an almost sufficient condition for an
itinerary to be realized as the second kneading sequence in the family $\G$. This is very
similar to the realization of maximal kneading sequences in unimodal families but the proof
involves some particularities of our family $\G$. For the convenience of the reader, we include 
a complete proof.
\begin{prop}
\label{propMin}
Let $\al\leq\al_0<\be_0\leq\be'$ and $\um$ be a minimal itinerary such that
$$\ku(\al_0)\prec \um\prec\ku(\be_0).$$
Then there exists $\ga\in(\al_0,\be_0)$ such that
$$\ku(\ga)=\um.$$
\end{prop}
\begin{proof}
Suppose that $\ku(\ga)\neq\um$ for all $\ga\in(\al_0,\be_0)$. Let $\ga_0=\sup\SetDef{\ga\in[\al_0,\be_0]}
{\ku(\ga)\preceq\um}$ and $n=\min\SetDef{j\geq 0}{\ku(\ga_0)(j)\neq\um(j)}<\ft$. Then, using the continuity
of $g_\ga^n$, $c_1$ and $c_2$ one may check that 
$$k_n=\ku(\ga_0)(n)\in\A_c=\{c_1,c_2\},$$
otherwise the maximality of $\ga_0$ is contradicted as $\ku(0),\ldots,\ku(n-1)$ and $\ku(n)$ would be 
constant on an open interval that contains $\ga_0$. There are two possibilities 
\begin{enumerate}
\item $k_n=c_1$ so $g_{\ga_0}^n(c_2)=c_1$ therefore $c_2$ is preperiodic.
\item $k_n=c_2$ so $g_{\ga_0}^n(c_2)=c_2$ therefore $c_2$ is super-attracting.
\end{enumerate}

Therefore $\ga_0>\al$ and $\ga_0\leq \be'<\be$. Let us recall that $\G$ is a natural family so 
the asymptotic kneading sequences $\ku^-(\ga_0)$ and $\ku^+(\ga_0)$ do exist and are infinite. 
Then the definition of $\ga_0$ shows that
\begin{equation}
\label{equM}
\min(\ku(\ga_0),\ku^-(\ga_0))\preceq\um\preceq\ku^+(\ga_0).
\end{equation}

Let $\um=m_0m_1\ldots m_n\ldots$ and $S=m_0\ldots m_{n-1}\in\A_I^n$ be the maximal common prefix
of $\ku(\ga_0)$ and $\um$, so $\ku(\ga_0)=Sc_j$ with $j\in\{1,2\}$. Therefore, using Proposition 
\ref{propAsy}, $m_n\in\{I_j,I_{j+1}\}$.

Suppose $k_n=c_1$ so $g_{\ga_0}^n(c_2)=c_1$. Lemma \ref{lemFin} and property \requ{GKn} show that 
the sequences $I_1I_3^kc_2$ and $I_2I_3^kc_2$ are realized as itineraries by all $g_\ga$ with 
$\ga\in[\al,\be']$ for all $k\geq 0$. Moreover 
$x(I_1I_3^kc_2)$ is strictly increasing in $k$ for all $\ga\in[\al,\be']$ 
and it is continuous in $\ga$. Analogously, $x(I_2I_3^kc_2)$ is strictly decreasing in $k$ for all $\ga\in[\al,\be']$ and it is continuous in $\ga$. Then by compactness and
by the continuity of $\ga\ra g_\ga^n$ and of $\ga\ra c_1$
$$\ku^-(\ga_0),\ku^+(\ga_0)\in S\times\{I_1,I_2\}\times I_3^\ft.$$
Therefore inequality (\ref{equM}) shows that
$$\min(c_1,I_1I_3^\ft)=I_1I_3^\ft\preceq \sigma^n\um\preceq I_2I_3^\ft=\max(c_1,I_2I_3^\ft).$$ 
But $\um\in\Sigma_0$ so 
$$I_1I_3^\ft\prec \sigma^n\um\prec I_2I_3^\ft$$
therefore $m_n=c_1$ as $I_1I_3^\ft=\max (I_1\times \Sigma)$ and $I_2I_3^\ft=\min (I_2\times \Sigma)$,
a contradiction.

Consequently $\ku(\ga_0)=Sc_2$ so $c_2(\ga_0)$ is super-attracting. Then by Corollary \ref{corRep} 
there is a neighborhood $J$ of $\ga_0$ such that $a(\ga)$ is a periodic attracting point of period 
$n$ for all $\ga\in J$, $\ga\ra a(\ga)$ is continuous and $a(\ga_0)=c_2(\ga_0)$. By Singer's Theorem 
\ref{thmSing}, $c_2$ is contained in the immediate basin of attraction $B_0(a(\ga))$ for all $\ga\in J$,
which is disjoint from $c_1$. Therefore, considering the local dynamics of $g_\ga^n$ on a neighborhood of
$a(\ga)$, $\ku(\ga)=\iu(g_\ga(a))$ is also periodic of period $n$ or 
finite of length $n$ for all $\ga\in J$. As the family $\G$ is natural, there exists $\ve>0$ such that 
$c_2$ is not periodic for all $\ga\in(\ga_0-\ve,\ga_0+\ve)\sm\{\ga_0\}$. Again by Corollary \ref{corRep},
$\ku(\ga)=\ku^-(\ga_0)$ for all $\ga\in(\ga_0-\ve,\ga_0)$ and $\ku(\ga)=\ku^+(\ga_0)$ for all 
$\ga\in(\ga_0,\ga_0+\ve)$. Then Proposition \ref{propAsy} shows that 
$$\ku^-(\ga_0),\ku^+(\ga_0)\in\{(SI_2)^\ft,(SI_3)^\ft\}.$$
Let $\um_1=\min((SI_2)^\ft,(SI_3)^\ft)$ and $\um_2=\max((SI_2)^\ft,(SI_3)^\ft)$ and 
$$K=\SetDef{\iu\in\Sigma}{\iu\mbox{ minimal and }\um_1\prec\iu\prec\um_2}.$$ 
As the sequences $Sc_2$,
$\ku^-(\ga_0)$ and $\ku^+(\ga_0)$ are all realized as a kneading sequence $\ku(\ga)$ with $\ga\in[\al,\be']$,
using inequality (\ref{equM}) it is enough to show that 
$$K=\{Sc_2\}.$$

Let $\iu\in K\sm\{Sc_2\}$ so
$$\iu\in S\times\{I_2,I_3\}\times\Sigma.$$

Suppose $\epsilon(S)=1$ so $\um_1=(SI_2)^\ft$ and $\um_2=(SI_3)^\ft$. 
Suppose $\iu(n)=I_2$, then as $\epsilon(SI_2)=-1$ and $\iu$ is minimal
$$\iu\preceq\sigma^n(\iu)\prec(SI_2)^\ft=\sigma^n(\um_1)=\um_1,$$
a contradiction.

Analogously, suppose $\iu(n)=I_3$, then for all $k\geq 1$
$$\iu\preceq\sigma^{kn}(\iu)\prec(SI_3)^\ft=\sigma^{kn}(\um_2),$$
so, by induction, $\iu=(SI_3)^\ft=\um_2\notin K$.

The case $\epsilon(S)=-1$ is symmetric so we may conclude that $K=\{Sc_2\}$ which
contradicts our initial supposition.
\end{proof}

Let us prove a complementary combinatorial property.
\begin{lemma}
\label{lemMin}
Let $S\in\A_I^n$ with $\ku(\al)\preceq\siti\preceq\ku(\be')$ and such that $\siti$ is minimal.
If $i_1i_2\ldots\in\Sigma$ and $i_1,i_2,\ldots\in\A\sm\{I_1\}$ then
$$\sit^ki_1i_2\ldots\in\Sigma\mbox{ is minimal for all }k\geq|S|.$$
\end{lemma}
\begin{proof}
Let $\iu=\sit^ki_1i_2\ldots\in\Sigma$, $n=|S|$ and $k\geq n$. Suppose there exists $j>0$ such that  
$$\sigma^j(\iu)\prec\iu.$$
As $SI_2^\ft\preceq\ku(\be')=I_1\ldots$
$$\iu\in I_1\times\Sigma.$$
Then $j<n$ and we set $m=\min\SetDef{p\geq 0}{\sigma^j(\iu)(p)\neq \iu(p)}$. Therefore $m\leq n-1$ so
$$\sigma^j(SI_2^\ft)\prec SI_2^\ft$$
as $\iu$ coincides with $SI_2^\ft$ on the first $2n$ symbols, a contradiction.
\end{proof}

Using relation (\ref{equGKn}), $\ku(\ga)=I_1\ldots$ so $I_2^kc_j\in\Sigma(\ku(\ga))$ for all $k\geq 0$, 
$j=1,2$ and $\ga\in[\al,\be']$. Then by Lemma \ref{lemFin} the maps 
$$\ga\ra p_k(\ga)=x(I_2^kc_1)(\ga)\mbox{ and }\ga\ra q_k(\ga)=x(I_2^kc_2)(\ga)$$
are uniquely defined and continuous on $[\al,\be']$ for all $k\geq 0$.
Let us recall that $g_\ga$ is decreasing on $I_2$ so
$$c_1\prec I_2c_2\prec I_2^2c_1\prec I_2^3c_2\prec\ldots\prec I_2^\infty
\prec\ldots\prec I_2^3c_1\prec I_2^2c_2\prec I_2c_1\prec c_2,$$
therefore
$$c_1=p_0<q_1<p_2<q_3<\ldots<r<\ldots<p_3<q_2<p_1<q_0=c_2$$
for all $\ga\in[\al,\be']$.

Let us show that $p_k\ra r$ and $q_k\ra r$ as $k\ra\infty$ for all $\ga\in[\al,\be']$. Let 
$$
\begin{array}{l}
r^-=\lim\limits_{k\ra\infty}p_{2k}=\lim\limits_{k\ra\infty}q_{2k+1}
\mbox{ and }\\
r^+=\lim\limits_{k\ra\infty}q_{2k}=\lim\limits_{k\ra\infty}p_{2k+1}.
\end{array}
$$
Suppose that $r^-<r^+$ then by continuity $g_\ga(r^-)=r^+$ and $g_\ga(r^+)=r^-$, as
$g_\ga(p_{k+1})=p_k$ and $g_\ga(q_{k+1})=q_k$ for all $k\geq 0$. Then $r^-$ and $r^+$
are periodic points of period $2$ and with itinerary sequence $I_2^\infty$, which contradicts 
Lemma \ref{lemPer}. By compactness
\begin{equation}
\label{equPQ}
p_k,q_k\ra r\mbox{ uniformly as }k\ra\ft.
\end{equation}

The following proposition shows that these convergences have a counterpart in the parameter space.
\begin{prop}
\label{propConv}
Let $S\in\A_I^n$ for some $n\geq 0$ be such that $SI_2^\ft$ is minimal and $\ku^{-1}(SI_2^\ft)$ 
is finite. Let $\al\leq\al_0<\be_0\leq\be'$ be such that $\ku(\al_0)\prec SI_2^\ft\prec\ku(\be_0)$ and
$S'=SI_2^{k+1}$ with $k\geq 0$ and such that $\epsilon(S')=1$. If $\iu_1=S'c_1$, $\iu_2=S'c_2$ and  
$k$ is sufficiently large then we may define
\begin{equation}
\label{equAlBe}
\begin{array}{l}
\ga_1=\max\left(\ku^{-1}(\iu_1)\cap(\al_0,\be_0)\right)\mbox{ and}\\
\ga_2=\min\left(\ku^{-1}(\iu_2)\cap(\ga_1,\be_0)\right)
\end{array}
\end{equation}
and then 
$$\lim\limits_{k\ra\ft}(\ga_2-\ga_1)=0.$$
\end{prop}
\begin{proof}
First let us remark that the condition $\epsilon(S')=1$ guarantees that
$$\iu_1\prec SI_2^\ft\prec\iu_2.$$
Using for example convergences \requ{PQ} and the bijective map $\iu(g_\al)$ defined by \requ{SZ} there exists 
$N_0>0$ such that for all $k\geq N_0$, $\ku(\al_0)\prec \iu_1\prec\iu_2\prec\ku(\be_0)$. 
Moreover, if $k\geq n$ then $\iu_1$ and $\iu_2$ are minimal, using \rlem{Min}. 

Therefore for $k\geq\max(N_0,n)$ we may apply Proposition \ref{propMin} to show that there exist
$\ga_1\in\ku^{-1}(\iu_1)\cap(\al_0,\be_0)$ and $\ga_2\in\ku^{-1}(\iu_2)\cap(\ga_1,\be_0)$. As $\iu_1$ 
and $\iu_2$ are finite and the family $\G$ is natural, $\ku^{-1}(\iu_1)$ and $\ku^{-1}(\iu_2)$ are finite.

We may apply again Proposition \ref{propMin} to see that 
$\ga_1$ is increasing to a limit $\ga^-$ as $k\ra\ft$. Again by Proposition \ref{propMin} and by 
the finiteness of $\ku^{-1}(SI_2^\ft)$ there exists
$$\gt=\max\left(\ku^{-1}(SI_2^\ft)\cap(\al_0,\be_0)\right)<\be_0\mbox{ and }\ga^-\leq\gt.$$
For the same reasons there is $N>0$ such that $\ga_2>\gt$ for all $k\geq N$, therefore $\ga_2$ becomes
decreasing and converges to some $\ga^+\geq \gt$.

Suppose that the statement does not hold, that is 
$$\ga^-<\ga^+.$$ 
The map $\iu(g_\al)\fd I{\Sigma_0}$ is bijective and order preserving and $p_i\ra r$, $q_i\ra r$ 
as $i\ra\ft$ therefore 
$$\SetDef{\iu\in\Sigma_0}{\iu_1\preceq\iu\preceq\iu_2\mbox{ for all }k>0}=\{SI_2^\ft\}.$$ 
Then the definitions of $\ga^-$ and $\ga^+$ imply that
$$\ku(\ga)=SI_2^\ft\mbox{ for all }\ga\in[\ga^-,\ga^+],$$
which contradicts the hypothesis. 
\end{proof}

From the previous proof we may also retain the following Corollary.
\begin{corollary}
\label{corLim}
Assume the hypothesis of the previous proposition. Then
$$\limi k{\ga_1}=\limi k{\ga_2}=\gt$$
and $\ku(\gt)=SI_2^\ft$.
\end{corollary}

We may also control the growth of the derivative on the second critical orbit in the setting of
the last proposition. In fact, letting $k\ra\ft$, the second critical orbit spends most of its
time very close to the fixed repelling point $r$. Therefore the growth of the derivative along
this orbit is exponential.

Let us also compute some bounds for the derivative along two types of orbits.
\begin{lemma}
\label{lemBou}
Let $[\ga_1,\ga_2]\se[\al,\be']$, $n\geq 0$, $S\in\A_I^n$ and $\iu_1,\iu_2\in S\ts$ with $\iu_1\prec\iu_2$ 
be finite or equal to $I_1^\ft$, $I_2^\ft$ or $I_3^\ft$. If $\iu_1$, $\iu_2$ are realized on 
$[\ga_1,\ga_2]$ then there exists $\te > 0$ such that
$$\te<\left|\left(g_\ga^j\right)'(x)\right|<\te^{-1}$$
for all $\ga\in[\ga_1,\ga_2]$, $x\in[x(\iu_1),x(\iu_2)]$ and $j\var 1n$.
\end{lemma}
\begin{proof}
Let us remark that $\iu(x)\in S\ts$ therefore $\left(g_\ga^j\right)'(x)\neq 0$ for all $\ga\in[\ga_1,\ga_2]$, 
$x\in[x(\iu_1),x(\iu_2)]$ and $j\var 1n$. As $x(\iu_1)$ and $x(\iu_2)$ are continuous by Lemmas \ref{lemFix} and \ref{lemFin},
the set
$$\SetDef{(\ga,x)\in\R^2}{\ga\in[\ga_1,\ga_2],x\in[x(\iu_1),x(\iu_2)]}$$
is compact. Therefore the continuity of $(\ga,x)\ra\left(g_\ga^j\right)'(x)$ for all $j\var 1n$ implies 
the existence of $\te$.
\end{proof}

The previous lemma helps us estimate the derivative of $g_\ga^n(x)$ on a compact interval of parameters 
if $\iu(x)\in I_j^n\ts$ and $n$ is sufficiently large. Let us denote 
$$I_j(n)(\ga)=\SetDef{x\in I_j}{g_\ga^k(x)\in I_j\mbox{ for all }k\var 1n}$$
for j=1,2,3, the interval of points of $I_j$ that stay in $I_j$ under $n$ iterations. 
Let also $s_j$ be the unique fixed point in $I_j$.
\def\int{{[\ga_1,\ga_2]}}
\def\gi{{\ga\in\int}}
\def\agi{{for all $\gi$ }}
\def\fe{{\frac\ve2}}
\begin{lemma}
\label{lemDF}
Let $[\ga_1,\ga_2]\se[\al,\be']$, $j\in\{1,2,3\}$ and $\ve>0$. Let also
$$\la_1 = \la_1(j) = \min_\gi\left|g_\ga'(s_j)\right|,$$
$$\la_2 = \la_2(j) = \max_\gi\left|g_\ga'(s_j)\right|.$$
There exists $N>0$ such that for all $k>0$, $\gi$ and $x\in I_j(\max(k,N))(\ga)$
$$\la_1^{k(1-\ve)}<\left|\left(g_\ga^k\right)'(x)\right|<\la_2^{k(1+\ve)}.$$
\end{lemma}
\begin{proof}
Let us first observe that by the definition \requ{Bp} of $\be'$
$$1<\la_1\leq\la_2.$$
\rlem{Fin} shows that the itinerary sequences $I_j^nc_1$, $I_j^nc_2$ are realized on $[\al,\be']$ for all
$n\geq 0$. We may easily obtain analoguous convergences to \requ{PQ} if $j\in\{1,3\}$, therefore
$$x(I_j^nc_1),x(I_j^nc_2)\ra s_j\mbox{ uniformly as }n\ra\ft.$$
Moreover $\pa I_j(n)(\ga)\se\{x(I_j^{n}c_1),x(I_j^{n}c_2),s_j\}$ \agi and $n \geq 0$.
By continuity of $s_j$ and of $(\ga,x)\ra g_\ga'(x)$ there exists $N_0>0$ such that
$$\la_1^{1-\frac \ve2}<\left|g_\ga'(x)\right|<\la_2^{1+\fe}$$
\agi and $x\in I_j(N_0)(\ga)$.

Using \rlem{Bou} there exists $\te>0$ such that
$$\te<\left|\left(g_\ga^i\right)'(x)\right|<\te^{-1}$$
\agi, $x\in I_j(N_0)(\ga)$ and $1\leq i\leq N_0$.
Let $N_1>0$ be such that 
$$\la_1^{N_1\fe}>\te^{-1}\la_2^{N_0(1+\ve)}$$
and set $N=N_0+N_1$. Let $k>N_1$ and $n=\max(N_1,k-N_0)$ then
$$\te\la_1^{n\left(1-\fe\right)}<\left|\left(g_\ga^k\right)'(x)\right|<\te^{-1}\la_2^{n\left(1+\fe\right)}$$
\agi and $x\in I_j(m)(\ga)$, where $m=\max(k,N)$. As $n\geq N_1$ and $1<\la_1\leq\la_2$
$$\la_1^{k(1-\ve)}<\left|\left(g_\ga^k\right)'(x)\right|<\la_2^{k(1+\ve)}$$
\agi and $x\in I_j(m)(\ga)$. If $k\leq N_1$ then $g_\ga^n(x)\in I_j(N_0)(\ga)$ for all $n\var 0{k-1}$ so
$$\la_1^{k(1-\ve)}<\la_1^{k\left(1-\fe\right)}<\left|\left(g_\ga^k\right)'(x)\right|<
\la_2^{k\left(1+\fe\right)}<\la_2^{k(1+\ve)}$$
\agi and $x\in I_j(m)(\ga)$.
\end{proof}
We may remark that if we assume the hypothesis of the previous lemma then $g_\ga^k$ is monotone on 
$I_j(m)(\ga)$ therefore \agi
\begin{equation}
\label{equInt}
\la_2(j)^{-k(1+\ve)}<|I_j(m)(\ga)|<\la_1(j)^{-k(1-\ve)}.
\end{equation}

Let $d_n\fd{[\al,\be']}{\R_+}$ be defined by
$$d_n(\ga)=\left|\left(g_\ga^n\right)'(v)\right|,$$
where $v=g_\ga(c_2)$ the second critical value.
As $\ga\ra v$ and $\ga\ra g_\ga^n$ are continuous, $d_n$ is continuous. The family
$\G$ is natural so $d_n$ has finitely many zeros for all $n\geq 0$.

\begin{corollary}
\label{corDer}
Assume the hypothesis of \rpro{Conv} and let $\la_0=|g_{\gt}'(r)|>1$. For all $0<\ve<1$ there
exists $N>0$ such that if $k\geq N$ then
$$\la_0^{(n+k)(1-\ve)}<d_{n+k}(\ga)<\la_0^{(n+k)(1+\ve)}\mbox{ for all }\ga\in[\ga_1,\ga_2].$$
\end{corollary}
 
\begin{proof}
Let us remark that $|\ku(\ga)|>n$ for all $\ga\in[\ga_1,\ga_2]$ therefore there exists 
$\te>0$ such that
$$\te < d_n(\ga) < \te^{-1}\mbox{ for all }\ga\in[\ga_1,\ga_2].$$
Using the previous argument and \rcor{Lim} there exists $N_0>0$ such that if $k\geq N_0$ then
$$\la_0^{k\left(1-\frac\ve2\right)}<\left|\left(g_\ga^k\right)'(v_n)\right|<\la_0^{k\left(1+\frac\ve2\right)}
\mbox{ for all }\ga\in[\ga_1,\ga_2].$$
Therefore it is enough to choose $N\geq N_0$ such that 
$$\la_0^{N\fe}>\te^{-1}\la_0^{n(1-\ve)}.$$
\end{proof}


\section{TCE does not imply RCE} 
\label{sectUHPnRCE}

In this section we consider a family $\G\fd{[\al,\be]}\Pc_2$ (see the definition of $\Pc_2$ at page \pageref{defP2}) satisfying all properties (\ref{equGRep}) to
(\ref{equGKn}) and Lemmas \ref{lemFix} and \ref{lemPer} \allg{}.
We build a decreasing sequence 
of families $\G_n:[\al_n,\be_n]\rightarrow \Pc_2$ with $\G_0=\G$, $\al_n\nearrow\gl$ and 
$\be_n\searrow\gl$ as $n\rightarrow\infty$. This means that $\G_n(\ga)=\G(\ga)$ for all 
$n\geq 0$ and $\ga\in[\al_n,\be_n]$. 
We obtain our counterexample as a limit $g_{\gl}=\G(\gl)=\G_n(\gl)$ for all $n\geq 0$.
For all $n\geq 0$ we choose two finite minimal itinerary sequences $\iu_1(n+1)$ and
$\iu_2(n+1)$ as in Proposition \ref{propConv} such that
$$\ku_2(\al_n)\prec\iu_1(n+1)\prec\iu_2(n+1)\prec\ku_2(\be_n).$$
We set $\al_{n+1}=\ga_1$ and $\be_{n+1}=\ga_2$. Choosing sufficient long sequences 
$\iu_1(n+1)$ and $\iu_2(n+1)$ we obtain the convergences $\al_n\ra\gl$ and $\be_n\ra\gl$ as
$n\ra\ft$.

Let $T_2(x)=x^3-3x$ be the second Chebyshev polynomial. Observe that $-2$, $0$ and $2$ are fixed 
and that the critical points $c_1=-1$ and $c_2=1$ are sent to $2$ respectively $-2$. Its Schwarzian
derivative $S(T_2)(x)=-\frac {4x^2+1}{(x^2-1)^2}$ is negative on $\mathbb R\setminus\{c_1,c_2\}$.
Let $h>0$ small and for each $\ga\in[0,h]$ two order preserving linear maps 
$P_\ga(x)=x(4+\ga)-2-\ga$ and $Q_\ga(y)=\frac{y-T_2(-2-\ga)}{2-T_2(-2-\ga)}$ 
that map $[0,1]$ onto $[-2-\ga,2]$ respectively $[T_2(-2-\ga),T_2(2)]$ onto $[0,1]$. Let then 
\begin{equation}
\label{equGGa}
g_\ga=Q_\ga\circ T_2\circ P_\ga
\end{equation}
be a bimodal degree $3$ polynomial. 
As $S(P_\ga)=S(Q_\ga)=0$ for all $\ga\in[0,h]$, using equality (\ref{equSchw}), one may 
check that 
$$S(g_\ga)<0\mbox{ on }I\setminus\{c_1(\ga),c_2(\ga)\}\mbox{ for all }\ga\in[0,h].$$
If we write  
\begin{equation}
\label{equCont}
g_\ga(x)=\sum_{k=0}^3a_k(\ga)x^k
\end{equation}
it is not hard to check that $\ga\rightarrow a_k(\ga)$ is continuous on $[0,h]$ for $k\var 03$ 
therefore $\ga\rightarrow g_\ga$ is continuous with respect to the $C^1$ topology on $I$. 
By the definition of $\Pc_2$ (see page \pageref{defP2}), as $g_\ga(0)=0$ for all $\ga\in[0,h]$,
$\G:[0,h]\rightarrow\Pc_2$ with 
$\G(\ga)=g_\ga$ for all $\ga\in[0,h]$ is a family of bimodal maps with negative Schwarzian 
derivative. Observe that $0$ and $1$ are fixed points for all $\ga\in[0,h]$ and that they are
repelling for $g_0$, with $g_0'(0)=g_0'(1)=9$, which is condition (\ref{equGRep}). Moreover,
$g_\ga(c_1)=1$ for all $\ga\in[0,h]$ thus $\G$ satisfies also (\ref{equGPre}). Observe that 
if $\ga\in[0,h]$ then $Q_\ga(-2)=0$ if and only if $\ga=0$, therefore condition
(\ref{equGAl}) is also satisfied by $\G$. We show that $\G$ is also natural and 
that any minimal sequence $SI_2^\ft$
with $S\in\A_I^n$ and $n\geq 0$ equals the second kneading sequence $\ku(\ga)$ for at most 
finitely many $\ga\in[0,h]$. This allows us to use all the results of the previous section 
for the family $\G$.

Let $G\fd{[0,h]\times [0,1]}{\R}$ be defined by
$$G(\ga,x)=g_\ga(x)\mbox{ for all }\ga\in[0,h]\mbox{ and }x\in[0,1].$$
Then
$$G(\ga,x)=\frac{P_1(\ga,x)}{P_2(\ga)},$$
where $P_1$ and $P_2$ are polynomials. Using definition (\ref{equGGa}), we may compute $P_2$ easily
$$P_2(\ga)=2-T_2(-2-\ga)=(\ga+1)^2(\ga+4).$$
We may therefore extend $G$ analytically on a neighborhood $\Omega\se\R^2$ of $[0,h]\times[0,1]$.
The critical points $c_1$ and $c_2$ are continuously defined on $[0,h]$ by Lemma \ref{lemFin}.
They are also analytic in $\ga$ as a consequence of
the \IFT for real analytic maps applied to $\frac{\partial G}{\partial x}$. 
Therefore for all $n\geq 0$ the map $g_\ga^n(c_2)$ is analytic on a neighborhood 
of $[0,h]$ so 
$$c_j(\ga)-g_\ga^n(c_2)\mbox{ has finitely many zeros in }[0,h]$$
for all $j\in\{1,2\}$ and $n\geq 0$ as $g_0^n(c_2)=0$ and $c_1(\ga),c_2(\ga)\in(0,1)$ for all 
$\ga\in[0,h]$. The family $\G$ is therefore natural so by eventually shrinking $h$ we may also
suppose that $\G$ satisfies property (\ref{equGKn}) and Lemmas \ref{lemFix} and \ref{lemPer}
for all $\ga\in[0,h]$. Then the repelling fixed point $r$ is continuously defined on $[0,h]$ and
again by the \IFT applied to $G(\ga,x)-x$, it is analytic on a neighborhood of $[0,h]$. Then 
$$r(\ga)-g_\ga^n(c_2)\mbox{ has finitely many zeros in }[0,h]$$
for all $n\geq 0$ as $r(0)-g_0^n(c_2)=\frac 12$.

Let then $\G_0=\G$ so $\al_0=0$ and $\be_0=h$. Our counterexample $g_{\gl}$ should be $TCE$ but 
not $RCE$ (see Definitions \ref{defRCE} and \ref{defUHP}). 
Its first critical point is non-recurrent as $g_\ga(c_1)=1$
and $1$ is fixed \allg 0. Therefore the second critical point $c_2$ should be recurrent and not
Collet-Eckmann. We let $c_2$ accumulate on $c_1$ also to control the growth of the derivative along
its orbit. We build $g_{\gl}$ such that its second critical orbit spends 
most of the time near $r$ or $1$ so its derivative accumulates sufficient expansion. We show that $g_{\gl}$
is $ExpShrink$ (thus $TCE$) using a telescopic construction, in an analogous way to the proof of \rthm{MCN1}.

\subsection{A construction}
\label{sectConstr}
The construction of the sequence $(\G_n)_{n\geq 0}$ is realized by imposing at the $n$-th step
the behavior of the second critical orbit for a time span $t_{n-1}+1,t_{n-1},\ldots,t_n$. This is achieved
specifying the second kneading sequence and using Proposition \ref{propConv}. We set $t_0=0$.

We have seen that $\ku^+(0)=I_1^\ft$ and that $g_\ga(x)>x$ for all $x\in(0,c_1)$ and all $\ga\in[0,h]$ as
$0$ is repelling and $g_\ga$ has no fixed point in $(0,c_1)$. Therefore the backward orbit of $c_1$ in $I_1$
converges to $0$ and by compactness the convergence is uniform. Then
$$\ku^{-1}\lr{I_1^kc_1}\ra 0\mbox{ as }k\ra\ft,$$
using \rpro{Min} for their existence. Then for any $\ez>0$ there is $k_0>0$ such that 
$I_1^{k_0}c_1\prec\ku(\be_0)$ and $||g_0-g_\ga||_{C^1}<\ez$ for all $\ga\in[0,\ku^{-1}(I_1^{k_0}c_1)]$.
In particular, if 
$$1<\la<\la'<\abs{g_0'(r)}=3<\abs{g_0'(0)}=\abs{g_0'(1)}=9$$
then for $\ez$ sufficiently small
\begin{equation}
\label{equR}
\la'<\abs{g_\ga'(r)},\la'<\abs{g_\ga'(0)}\mbox{ and }\la'<\abs{g_\ga'(1)}
\end{equation} 
for all $\ga\in[0,\ku^{-1}(I_1^{k_0}c_1)]$.
Let $S_0=I_1^{k_0+1}\in\A_I^{k_0+1}$ so $\iu\prec I_1^{k_0}c_1$ for all $\iu\in S_0\times\Sigma$. Moreover,
$S_0I_2^\ft$ is minimal. Using \rpro{Conv} we find $\al_0<\ga_1<\ga_2<\be_0$ such that
$$\ku(\al_0)\prec\ku(\ga_1)\prec S_0\iti\prec\ku(\ga_2)\prec\ku(\be_0)$$
with $\ku(\ga_1),\ku(\ga_2)\in S_0I_2\ts$ and
$$|\ga_2-\ga_1|<2^{-1}.$$
We set $\al_1=\ga_1$ and $\be_1=\ga_2$ and define $\G_1\fd{\iab 1}{\Pc_2}$ by $\G_1(\ga)=\G(\ga)=g_\ga$
\allg 1.
Moreover, let $t_1=k+|S_0|$ and $S_1=S_0I_2^k$, where $k$ is specified by \rpro{Conv}, then  
$$\ku(\ga)\in S_1I_2\ts,$$
\allg 1. Using \rcor{Der} we may also assume that 
\begin{equation}
\label{equDn}
d_m(\ga)>\la^m,
\end{equation}
\allg 1, where $m=t_1=|S_1|$. Let us recall that $d_n(\ga)=\abs{\left(g_\ga^n\right)'(v)}$ and $v=g_\ga(c_2)$.

Then we build inductively the decreasing sequence of families $(\G_n)_{n\geq 0}$ such that for all $n\geq 1$,
$\G_n$ satisfies 
\begin{equation}
\label{equKn}
\ku(\ga)\in S_nI_2\ts,
\end{equation}
\begin{equation}
\label{equFin}
|\ku(\al_n)|,|\ku(\be_n)|<\ft,
\end{equation}
\begin{equation}
\label{equConv}
|\be_n-\al_n|<2^{-n},
\end{equation}
\begin{equation}
\label{equSi}
\ku(\al_n)\prec S_nI_2^\ft\prec\ku(\be_n),
\end{equation}
and conditions \requ{R} and \requ{Dn} \allg n, for some $S_n\in\A_I^m$ with $S_nI_2^\ft$ minimal, 
where $m=t_n$. As the sequence $(\G_n)_{n\geq 0}$ is decreasing, inequality \requ{R} is satisfied
by all $\G_n$ with $n\geq 1$. For transparency we denote $v_n=g_\ga^n(v)$ and
$$d_{n,p}(\ga)=\left|\left(g_\ga^p\right)'(v_n)\right|,$$
which also equals $d_{n+p}(\ga)d_n^{-1}(\ga)$, whenever $|\ku(\ga)|>n$ so $d_n(\ga)\neq 0$.

Let us describe two types of steps, one that takes the second critical orbit near $c_1$ to 
control the growth of the derivative and the other that takes it near $c_2$ to make the 
second critical point $c_2$ recurrent. We alternate the two types of steps in the construction of the
sequence $(\G_n)_{n\geq 0}$ to obtain our counterexample.

The following proposition describes the passage near $c_1$.
\begin{prop}
\label{propA}
Let the family $\G_n$ with $n\geq 1$ satisfy conditions \requ{R} to \requ{Si} and 
$$0<\la_1<\la_2<\la.$$
Then there exists a subfamily $\G_{n+1}$ of $\G_n$ satisfying the same conditions and such
that there exists $2t_n<p<t_{n+1}$ with the following properties
\begin{enumerate}
\item $\max\limits_{\ga\in[\al_{n+1},\be_{n+1}]}\abs{\log\abs{g_\ga'(r)}-\frac 1{p-1}\log d_{p-1}(\ga)}
<\log \la_2-\log\la_1$.
\item $\la_1^p<d_p(\ga)<\la_2^p$ \allg{n+1}.
\item $d_{t_n,l}(\ga)>\la^l$ \allg{n+1} and $l\var 1{p-1-t_n}$.
\item $d_{p,l}(\ga)>\la^l$ \allg{n+1} and $l\var 1{t_{n+1}-p}$.
\item $d_{t_n,t_{n+1}-t_n}(\ga)>\la^{t_{n+1}-t_n}$ \allg{n+1}.
\end{enumerate}
\end{prop}
\begin{proof}
This proof follows a very simple idea, to define the family $\G_{n+1}$ with
$$S_{n+1}=S_nI_2^{k_1+1}I_3^{k_2}I_2^{k_3},$$
as described by properties \requ{Kn} and \requ{Si}.
For $k_1$ and $k_3$ sufficiently large there exist $k_2$ such that the conclusion is satisfied
for $p=t_n+k_1+1$.

Let us apply \rpro{Conv} to $S_n$, $\al_n$ and $\be_n$. Let $k_1=k+1$, $\la_0=\abs{g_{\gt}'(r)}$ and 
$\la_3=\abs{g_{\gt}'(1)}$.  By inequality \requ{R} 
$$0<\la_1<\la_2<\la<\la_0$$ 
therefore there exists $\ve_0\in(0,1)$ such that 
$$\frac{(1+\ve_0)\log\la_0-\log\la_2}{(1-\ve_0)\log\la_3}<\frac{(1-\ve_0)\log\la_0-\log\la_1}{(1+\ve_0)\log\la_3}.$$
We choose $0<\ve<\ve_0$ such that
$$\ve<\frac {\log\la_2-\log\la_1}{8\log\la_0}.$$
Let us recall that  
\begin{equation}
\label{equSn}
\ku(\ga_1)=S_nI_2^{k_1}c_1\prec S_{n+1}\ts\prec\ku(\ga_2)=S_nI_2^{k_1}c_2.
\end{equation}
Using \rlem{DF} and \rcors{Lim}{Der} there exists $N_0$ such that if $k_1>N_0$ then
the first and the third conclusions are satisfied provided $[\al_{n+1},\be_{n+1}]\se[\ga_1,\ga_2]$.

Let $y(\ga)\in I$ with $\iu(y)\in I_2I_3^{k_2}I_2\ts$ and $y'=g_\ga(y)$. 
By \rcor{Lim}, \rlem{DF} and inequality \requ{Int} there exist $N_1,N_0'>0$ such that if $k_1>N_1$ and 
$k_2> N_0'$ then for all $\ga\in[\ga_1,\ga_2]$
\begin{equation}
\label{equDist}
\la_3^{-k_2(1+\ve)}<|1-y'|<\la_3^{-(k_2-1)(1-\ve)},
\end{equation}
as $y\in I_3(k_2-1)\sm I_3(k_2)$.

Let us recall that $g_\ga(x)=\sum_{k=0}^3a_k(\ga)x^k$ with $a_i$ continuous and $g_\ga'(c_1)=0$, 
$g_\ga''(c_1)\neq 0$ for all $\ga\in[\al,\be']$ and $c_1$ is continuous. Therefore there exist constants
$M>1$, $\de>0$ and $N_2>0$ such that if $k_1>N_2$ and $\ga\in[\ga_1,\ga_2]$ then
\begin{equation}
\label{equPow}
\begin{array}{lcccl}
 M^{-1}(x-c_1)^2 & < & |1-g_\ga(x)| & < & M(x-c_1)^2 \mbox{ and} \\
 M^{-1}(x-c_1) & < & \abs{g_\ga'(x)} & < & M(x-c_1)
\end{array}
\end{equation}
for all $x\in(c_1-\de,c_1+\de)$. Using inequality \requ{Dist} there exists $N_1'$ such that if $k_2>N_1'$
then $|1-y'|=|1-g_\ga(y)|<M^{-1}\de^2$ so $|y-c_1|<\de$, therefore
$$M^{-\frac 32}\la_3^{-\frac{k_2}2(1+\ve)}<\abs{g_\ga'(y)}<M^{\frac 32}\la_3^{-\frac{k_2-1}2(1-\ve)}.$$ 

Let $k_1>\max(t_n,N_0,N_1,N_2)$ and $k_2>\max(N_0',N_1')$. \rlem{Min} shows that $S_{n+1}I_2^\ft$ 
is minimal. We may therefore apply \rpro{Conv} with $S=S_nI_2^{k_1+1}I_3^{k_2}$, using inequality 
\requ{Sn}. Let $k_3=k$ and $\al_{n+1}$ and $\be_{n+1}$ be the new bounds for $\ga$ provided by \rpro{Conv}. 
Let us recall that $p=t_n+k_1+1$ and $v_n=g_\ga^{n+1}(c_2)$ for all $n\geq 0$, therefore 
$\iu(v_{p-1})\in I_2I_3^{k_2}\ts$ so we may set $y=v_{p-1}$ and $y'=v_{p}$. Let us remark that
$$d_p(\ga)=d_{p-1}(\ga)\cdot\abs{g_\ga'(y)}\ag{n+1}.$$
By \rcor{Der}, if $k_1$ is sufficiently large, then \allg{n+1}
$$M^{-\frac 32}\la_0^{(p-1)(1-\ve)}\la_3^{-\frac{k_2}2(1+\ve)} < d_p(\ga) <
M^{\frac 32}\la_0^{(p-1)(1+\ve)}\la_3^{-\frac{k_2-1}2(1-\ve)}.$$
Therefore the second conclusion is satisfied if
$$p\log\la_1<-\frac 32\log M+(p-1)(1-\ve)\log\la_0-\frac{k_2}2(1+\ve)\log\la_3$$ and
$$p\log\la_2>\frac 32\log M + (p-1)(1+\ve)\log\la_0-\frac{k_2-1}2(1-\ve)\log\la_3.$$
We may let $p\ra\ft$ and $\frac{k_2}{2p}\ra\eta$ so it is enough to find $\eta > 0$ such that
$$
\begin{array}{rcl}
 \log\la_1 & < & (1-\ve)\log\la_0-\eta(1+\ve)\log\la_3 \mbox{ and}\\
 \log\la_2 & > & (1+\ve)\log\la_0-\eta(1-\ve)\log\la_3.
\end{array}
$$
The existence of $\eta$ is guaranteed by the choice of $\ve<\ve_0$.

Again by inequality \requ{R}, \rlem{DF} and \rcor{Der}, if $k_2$ and $k_3$ are sufficiently large then
the last two conclusions are satisfied. If $k_3$ is sufficiently large then by \rcor{Lim} inequality
\requ{Conv} is also satisfied.
\end{proof}

The following proposition describes the passage of the second critical orbit near $c_2$.
\begin{prop}
\label{propB}
Let the family $\G_n$ with $n\geq 1$ satisfy conditions \requ{R} to \requ{Si} and 
$$\De>0.$$
Then there exists a subfamily $\G_{n+1}$ of $\G_n$ satisfying the same conditions and such
that there exists $t_n<p<t_{n+1}$ with the following properties
\begin{enumerate}
\item $\abs{g_\ga^p(c_2)-c_2}<\De$ \allg{n+1}.
\item $d_{t_n,l}(\ga)>\la^l$ \allg{n+1} and $l\var 1{t_{n+1}-t_n}$.
\item $d_{p-1,t_{n+1}-p+1}(\ga)>\la^{t_{n+1}-p+1}$ \allg{n+1}.
\end{enumerate}
\end{prop}
\begin{proof}
Once again, we build the family $\G_{n+1}$ using the prefix of the kneading sequence
$$S_{n+1}=S_nI_2^{k_1}S_nI_2^{k_2+1}I_3I_2^{k_3},$$
and show that we may choose $k_2$ such that if $k_1$ and $k_3$ are sufficiently large then the conclusion
is satisfied for $p=t_n+k_1$.

We apply \rpro{Conv} to $S_n$, $\al_n$ and $\be_n$. Let $k_1=k+2$, $\la_0=\abs{\gg'(r)}>\la'$ and
$$S'=S_nI_2^{k_2+1}I_3.$$
In the sequel $k_2$ is chosen such that $\epsilon(S')=1$ therefore $\ku(\gt)=S_n\iti\prec S'\ldots$, so
$$S_{n+1}\iti\mbox{ is minimal if }k_1-1>k_2>t_n.$$
Indeed, suppose that there exists $j>0$ such that $\si^j\left(S_{n+1}\iti\right)\prec S_{n+1}\iti$. 
Let us recall that $t_n=|S_n|$ and $S_n\iti$ is minimal, using property \requ{Si} of $\G_n$. 
A similar reasoning to the proof of \rlem{Min} shows that $j$ can only be equal to $t_n+k_1$ so
$$S'\ldots\prec S_nI_2^{k_1}\ldots$$
which contradicts $S_n\iti\prec S'\ldots$, as $k_1\geq k_2+2$. Moreover,
$$\ku(\ga_1)=S_nI_2^{k_1-1}c_1\prec S_{n+1}\iti\prec \ku(\ga_2)=S_nI_2^{k_1-1}c_2,$$
and $\iu'=I_2^2S'c_1\prec I_2^2S'c_2=\iu''\prec c_2$ are realized for all $\ga\in[\ga_1,\ga_2]$, using \rlem{Fin}. 
Let us remark that $\gg$ has no homterval as $v_{t_n}=r$, using Singer's \rthm{Sing}. Therefore 
$$\lim\limits_{k_2\ra\ft}\gg(x(\iu''))=c_2$$
as $\gg(x(\iu''))=x(\si\iu'')<c_2$ is increasing with respect to $k_2$ and 
$$\SetDef{\iu\in\Sigma_0}{I_2S'c_2\prec\iu\prec c_2\mfa k_2>0}=\emptyset.$$ 
Let $k_2$ be such that $|c_2(\gt)-x(\si\iu'')(\gt)|<\Delta$.
Using \rcor{Lim} and the continuity of $c_2$ and of $x(\si\iu'')<x(\si\iu')<c_2$ there exists $N_0>0$ 
such that if $k_1>N_0$ then 
\begin{equation}
\label{equDE}
|c_2-x|<\Delta,
\end{equation}
for all $\ga\in[\ga_1,\ga_2]$ and $x\in[x(\si\iu''),x(\si\iu')]$.
\rlem{Bou} applied to $\iu'$ and $\iu''$ yields $\te>0$ such that if $l=t_n+k_2+4$ then
\begin{equation}
\label{equTe}
\te<\left|\left(g_\ga^j\right)'(x)\right|<\te^{-1},
\end{equation}
for all $\ga\in[\ga_1,\ga_2]$, $x\in[x(\iu'),x(\iu'')]$ and $j\var 1l$. 
\rlem{DF} provides $N_1>0$ such that if $k_1>N_1$ then
\begin{equation}
\label{equLa}
\left(\la'\right)^j<d_{t_n,j}(\ga)\mfa\ga\in[\ga_1,\ga_2]\mfa j\var 1{k_1-2}.
\end{equation}
As $\la'>\la$ there exists also $N_2>0$ such that
\begin{equation}
\label{equTE}
\te^{-1}\la^{N_2-2+l}<(\la')^{N_2-2}.
\end{equation}
Let $k_1>\max(k_2+1,N_0,N_1,N_2)$ and $S''=S_nI_2^{k_1}S'$. Let us remark that 
$S''\iti=S_{n+1}\iti$ thus we may apply \rpro{Conv} to $S''$, $\ga_1$ and $\ga_2$. Let $\al_{n+1}$ 
and $\be_{n+1}$ be the new bounds for $\ga$ provided by \rpro{Conv} and $k_3=k$. 

As $g_\ga^p(c_2)=v_{p-1}$ and $\si^{p-1}\left(S''I_2^{k_3}\ldots\right)=I_2S'I_2\ldots$
$$g_\ga^p(c_2)\in[x(\si\iu''),x(\si\iu')],$$
for all $\ga\in[\al_{n+1},\be_{n+1}]$ thus, by inequality \requ{DE}, the first conclusion is satisfied. Moreover, using 
inequalities \requ{Te}, \requ{La} and \requ{TE}
$$\la^j<d_{t_n,j}(\ga)\ag{n+1}\mbox{ and }j\var 1{\abs{S''}=k_1-2+l}.$$
Using \rlem{DF} and \rcor{Der}, for $k_3$ sufficiently large the last two conclusions are satisfied. 
If $k_3$ is sufficiently large then by \rcor{Lim} inequality \requ{Conv} is also satisfied by $\al_{n+1}$ and $\be_{n+1}$.
\end{proof}

\subsection{Some properties of polynomial dynamics}
\label{sectRat}
Let us introduce some notations. For any set $E\subseteq\overline{\mathbb C}$ and $\al>0$,
we define the $\alpha$-neighborhood of $E$ by
$$E_{+\alpha}=B\left(E,\alpha\right) =\{x\in\overline{\mathbb C}|\MOo{dist}{x,E}<\alpha\}.$$
One may easily check that if $f,g\fd\Omega\RS$ with $\Omega\se\RS$ and $\de>||f-g||_\ft$ then for all 
$B\se\RS$ 
\begin{equation}
\label{equInc}
g^{-1}(B)\se f^{-1}(B_{+\de}).
\end{equation}
\def\Cd{{\C_d[z]}}
Using this simple observation we show that in a neighborhood of an $ExpShrink$ polynomial
some weaker version of Backward Stability is satisfied, see \rpro{Free}.

\begin{definition}
We say that a rational map $f$ has \emph{Backward Stability} if
for any $\varepsilon>0$ there exists $\delta>0$ such that for all $z\in
J$, the Julia set of $f$, all $n\geq 0$ and every connected component $W$ of $f^{-n}(B(z,\delta))$
$$\MO{diam}{W}<\varepsilon.$$
\label{defBS}
\end{definition}

Let us first show 
that the Julia set is \emph{continuous} in the sense of \rlem{JC}. For transparency we introduce additional 
notations. We denote by $\Cd$ the space of complex polynomials of degree $d$. If $f(z)=\sum_{i=0}^da_iz^i\in\Cd$ 
let us also denote
$$|f|=\max\limits_{0\leq i\leq d}|a_i|.$$
By convention, when $f\in\Cd$ and we compare it to another polynomial $g$ writing $|f-g|$ we also assume that
$g\in\Cd$.

Let us observe that the coefficients of $f^n=f\circ f\circ\ldots\circ f$, the $n$-th iterate of $f$, are 
continuous with respect to $(a_0,a_1,\ldots,a_d)\in\R^{d+1}$ for all $n>0$. Therefore given $f\in\Cd$, $m>0$ and 
$\ve>0$ there exists $\de>0$ such that if $|f-g|<\de$ then
$$\left|f^i-g^i\right|<\ve\mbox{ for all }i\var 1m.$$

Given a compact $K\se\C$, the map $\R^{d+1}\ni(a_0,a_1,\ldots,a_d)\ra f\in\Cd$ is continuous with respect 
to the topology of $C(K,\C)$. Therefore for all $f\in\Cd$, $\ve>0$ and $m>0$ there exists $\de>0$ such 
that if $|f-g|<\de$ then 
\begin{equation}
\label{equNorm}
\nm{f^i-g^i}_{\ft,K}<\ve\mbox{ for all }i\var 1m.
\end{equation}

\begin{lemma}
\label{lemJC}
Let $f\in\Cd$ with $d\geq 2$ and such that its Fatou set is connected and let $J$ be its Julia set. 
For all $\ve>0$ there exists $\de>0$ such that if $|f-g|<\de$ then
$$J_g\se J_{+\ve}.$$
\end{lemma}
\begin{proof}
The Fatou set of $f$ is the basin of attraction of $\ft$ and $J$ is compact and invariant. Let
$|J|=\max\limits_{z\in J}|z|$, then for all $M\geq |J|$
$$J=\SetDef{z\in\C}{\left|f^n(z)\right|\leq M\mbox{ for all }n\geq 0}.$$
Let $f(z)=\sum_{i=0}^da_iz^i\in\Cd$. There exists $R>1$ such that if $|f-g|<\frac 12|a_d|$ then
$$|J_g|\leq R.$$
Indeed, it is enough to choose 
$$R>4d+2|a_d|^{-1}\left(1+\sum\limits_{i=0}^{d-1}|a_i|\right),$$
and check that if $|z|>R$ then $|g(z)|>|z|+1$.

Let $T=\SetDef{z\in\RS}{\MO{dist}(z,J)\geq\ve}$. As $T$ is compact in $\RS$ and contained in the basin of attraction of $\ft$,
there is $m>0$ such that 
$$\left|f^m(z)\right|>R+1\mfa z\in T.$$
Let $K=\overline{B(0,R+1)}$ a compact such that $J_{+\ve},J_g\se K$ if $|f-g|<\frac 12|a_d|$. Inequality
\requ{Norm} yields $0<\de<\frac 12|a_d|$ such that if $|f-g|<\de$ then
$$\nm{f^i-g^i}_{\ft,K}<1\mfa i\var 1m.$$
Therefore by the definitions of $R$ and $m$, if $|f-g|<\de$ then
$$\left|g^m(z)\right|>R\mfa z\in T,$$
thus $J_g\cap T=\emptyset$.
\end{proof}
\begin{remark}
The hypothesis $f$ polynomial and its Fatou set connected are somewhat artificial, introduced for
the elegance of the proof. It may be easily generalized to rational maps with attracting periodic orbits but without parabolic periodic orbits nor rotation domains.
\end{remark}

\begin{prop}
\label{propFree}
Let $f$ be an $ExpShrink$ polynomial satisfying the hypothesis of \rlem{JC}. There 
exists $\de>0$ such that for all $0<r<\de$ there exist $N>0$ and $d>0$ such that 
for all $g$ with $|f-g|<d$ and $z\in J_g$ 
$$\MO{diam}\MO{Comp} g^{-N}(B(z,\de))<r.$$
\end{prop}
We use the notation $\MO{Comp} A$ for connected components of the set $A$. The previous statement means that the inequality holds for any such component.
\begin{proof}
Let us denote $J$ the Julia set of $f$.
Let $r_0>0$ and $\la_0>1$ be provided by \rdef{ES} such that for all $z\in J$
$$\MO{diam}\MO{Comp} f^{-n}\left(B(z,r_0)\right)<\la_0^{-n}\mfa n\geq 0.$$

Let $\de=\frac {r_0}4$ and choose $N\geq 1$ such that 
$$\la_0^{-N}<r.$$
Inequality \requ{Norm} provides $d_0$ such that if $|f-g|<d_0$ then
$$\left|f^N(z)-g^N(z)\right|<\de\mfa z\in\overline{J_{+r_0}}.$$
\rlem{JC} yields $d_1>0$ such that if $|f-g|<d_1$ and $z\in J_g$ then there exists $z'\in J$ 
such that $|z-z'|<2\de$ therefore 
$$B\left(z,2\de\right)\se B\left(z',r_0\right).$$
We choose $d=\min(d_0,d_1)$ and $g\in\Cd$ with $|f-g|<d$. Using inequality \requ{Inc}
$$\MO{diam}\MO{Comp} g^{-N}\left(B(z,\de)\right)<\la_0^{-N}<r\mfa z\in J_g.$$
\end{proof}

\begin{corollary}
\label{corFree}
Let $f$ satisfy the hypothesis of \rpro{Free} and $\ve>0$. There exist $d,\de>0$
such that if $|f-g|<d$ then for all $z\in J_g$ and $n\geq 0$
$$\MO{diam}\MO{Comp} g^{-n}(B(z,\de))<\ve.$$
\end{corollary}
\begin{proof}
Let us use the notations defined by the proof of \rpro{Free}.
It is straightforward to check that $f$ has Backward Stability and that, 
by eventually decreasing $r_0$, we may also suppose  
$$\MO{diam}\MO{Comp} f^{-n}\left(B(z,r_0)\right)<\ve\mfa z\in J\mbox{ and }n\geq 0.$$
Let $m\geq 1$ such that 
$$\la_0^{-m}<\de.$$
Inequality \requ{Norm} provides $d_0$ such that if $|f-g|<d_0$ then
$$\left|f^i(z)-g^i(z)\right|<\de\mfa z\in\overline{J_{+r_0}}\mbox{ and }i\var 1m.$$
Let $d_1$, $d$ and $g$ be as in the proof of \rpro{Free}.
By inequality \requ{Inc}, for all $z\in J_g$
$$ \MO{diam}\MO{Comp} g^{-m}\left(B(z,\de)\right) < \de$$ and
$$ \MO{diam}\MO{Comp} g^{-i}\left(B(z,\de)\right) < \ve\mfa i\var 0m.$$
For some $z\in J_g$, let $W\in\MO{Comp}g^{-m}\left(B(z,\de)\right)$ and
$z_1\in W\cap J_g$. Then
$$W\se B(z_1,\de)$$
and the proof is completed by induction.
\end{proof}

Let us show that the hypothesis of \rlem{JC} is easy to check for polynomials in $\G_0$.
\begin{lemma}
\label{lemAcc}
If $g_\ga\in\G_0$ and its second critical orbit $\sqnz v$ accumulates on a repelling
periodic orbit then $g_\ga$ satisfies the hypothesis of \rlem{JC}. Moreover, if $\sqnz v$
is preperiodic then $g_\ga$ has $ExpShrink$.
\end{lemma}
\begin{proof}
By Theorems III.2.2 and III.2.3 in \cite{CG}, the immediate basin of attraction of an attracting or parabolic
periodic point contains a critical point. But $c_1$ is strictly preperiodic and $\sqnz v$ accumulates on a
repelling periodic orbit thus it cannot converge to some attracting or parabolic periodic point.
Using Theorem V.1.1 in \cite{CG} we rule out Siegel disks and Herman rings as their boundary should be
contained in the closure of the critical orbits which is contained in $[0,1]$ for all $g_\ga\in\G_0$.
Using Sullivan's classification of Fatou components, Theorem IV.2.1 in \cite{CG}, the Fatou set
equals the basin of attraction of infinity which is connected for all polynomials by the maximum
principle.

If $\sqnz v$ is preperiodic then $g_\ga$ is semi-hyperbolic therefore by the main result in \cite{JJ} or by \rthm{MCN1} it has $ExpShrink$.
\end{proof}

Let us recall some general distortion properties of rational maps. 
The following result is a classical distortion estimate due to Koebe, see for example Lemma 2.5 in \cite{NUH}.

\begin{lemma}[Koebe]
Let $g\fd B\C$ be a univalent map from the unit disk into the complex plane. Then the image
$g(B)$ contains the ball $B\lr{g(0),\frac 14\abs{g'(0)}}$. Moreover, for all $z\in B$ we have that
$$\frac{\lr{1-|z|}}{\lr{1+|z|}^3}\leq\frac{\abs{g'(z)}}{\abs{g'(0)}}\leq\frac{\lr{1+|z|}}{\lr{1-|z|}^3},$$
and
$$\abs{g(z)-g(0)}\leq\abs{g'(z)}\frac{|z|(1+|z|)}{1-|z|}.$$
\label{lemKoebe}
\end{lemma}

For the remainder of this section, let $f$ be any rational map and $\mathrm{Crit}$ the set of critical points of $f$.
\begin{description}
\item[Distortion.]
This is a reformulation of the previous lemma.
For all $D>1$ there exists $\varepsilon>0$ such
that if the open $W$ satisfies
\begin{equation}
\MO{diam}{W}\leq\varepsilon\MOo{dist}{W, \mathrm{Crit}},
\label{equFar}
\end{equation}
then the distortion of $f$ in $\overline{W}$ is bounded by $D$.
\item[Pullback.] Once a sufficiently small $r>0$ is fixed, there exists $M\geq
1$ such that for any open $U$ with $\MO{diam}{U}\leq r$ and for all $W\in\MO{Comp}f\mo(U)$ and all $z\in
\overline{W}$
\begin{equation}
\MO{diam}{W}\leq M|f'(z)|^{-1}\MO{diam} U.
\label{equMD}
\end{equation}
We shall use this estimate for $W$ close to $\mathrm{Crit}$.
\end{description}

\subsection{A counterexample}

Using \rpros AB we build a sequence of families $\sqno\G$ which converge to a bimodal polynomial $g$
that has $ExpShrink$. Its first critical point $c_1$ is non-recurrent
as $g(c_1)=1$ and $1$ is a repelling fixed point. The second critical point $c_2$ is recurrent and it does not
satisfy the Collet-Eckmann condition. Therefore $g$ does not satisfy $RCE$.

We obtain the following theorem which states that the converse of \rthm{MCN1} does not hold. We use the equivalence of $TCE$ and $ExpShrink$ \cite{ETIC}.
\begin{thma}
There exists an $ExpShrink$ polynomial that is not $RCE$.
\end{thma}
The proof that $g$ has $ExpShrink $ is analogous to that of \rthm{MCN1}. 
This paper contains a complete proof of Theorem A. However, remarks about the proof of \rthm{MCN1} are present for the convenience of the reader. 
As $g$ is not $RCE$ we have to 
modify some of the tools like Propositions 9, 10 and 11 in \cite{MCN1}. The polynomial $g_0$ is
Collet-Eckmann and semi-hyperbolic thus $RCE$. By the main result of \cite{JJ} or by \rthm{MCN1}, $g_0$ has $ExpShrink$. Choosing
the family $\G_1$ in a sufficiently small neighborhood of $g_0$ we show two contraction results
similar to Propositions 9 and 10 in \cite{MCN1} that hold on $\G_1$, \rcor{GNR} and \rpro{GCE} below.
As $g\in\G_1$ we may choose constants $\mu,\te,\ve$, $R$ and $N_0$ - as described in the sequel - that do not depend on $g$. 

The main idea of the proof of Theorem A is that in inequality \requ{MD} the right term may be
much larger than the left term, see also \rlem{Pull}. This means that when pulling back a ball $B$ to 
$B^{-1}$ near a degree two critical point, the diameter of $B^{-1}$ is comparable to the square 
root of the radius of $B$ but $\abs{f'(z)}\mo$ may be as large as one wants for some $z\in B\mo$. This 
is the main difference between growth conditions in terms of the derivative or in terms of the 
diameter of pullbacks.

\rcor{GBS}, an immediate consequence of \rcor{Free}, replaces Proposition 11 in \cite{MCN1} in the proof of Theorem A.
\begin{corollary}
\label{corGNR}
There exists $\de>0$ such that for all $0<r<R\leq\de$ there exist $\be>\al_0$ and $N>0$ such
that for all $\ga\in[\al_0,\be]$ and $z\in J$ the Julia set of $g_\ga$
$$\MO{diam}\MO{Comp} g_\ga^{-N}(B(z,R))<r.$$
\end{corollary}
\begin{proof}
Using \rlem{Acc}, $g_0$ satisfies the hypothesis of \rpro{Free}. Using the continuity of coefficients
of $g_\ga$ \requ{Cont} there exists $\be>\al_0$ such that
$$|g_0-g_\ga|<d\mfa \ga\in[\al_0,\be].$$
\end{proof}

The following consequence of \rcor{Free} is a weaker version of uniform Backward Stability. The proof is
analogous to the proof of the previous proposition.
\begin{corollary}
\label{corGBS}
For all $\ve>0$ there exist $\be>\al_0$ and $\de>0$ such that for all $\ga\in[\al_0,\be]$ and
$z\in J$ the Julia set of $g_\ga$
$$\MO{diam}\MO{Comp} g_\ga^{-n}(B(z,\de))<\ve\mfa n\geq 0.$$
\end{corollary}

Let us compute an estimate of the diameter of a pullback far from critical points.
\begin{lemma}
\label{lemPBDia}
Let $h\fd {B(z,2R)}\C$ be an analytic univalent map and $U\ni z$ a connected open with $\dia U\leq R$.
If
$$\sup\limits_{x,y\in B(z,2R)}\abs{\frac{h'(x)}{h'(y)}}\leq D$$
then 
$$\dia U\leq D\abs{h'(z)}\mo\dia h(U).$$
\end{lemma}
\begin{proof}
Let $x,y\in\partial U$ such that $\abs{x-y}=\dia U$. Let $a=h(x)$, $b=h(y)$ and consider the pullback of the 
line segment $[a,b]$ that starts at $x$. Then there exists $t_0\in(0,1]$ such that 
$$[a,t_0a+(1-t_0)b]\se h\lr{B(z,2R)}$$
and such that the length of $h\mo\lr{[a,t_0a+(1-t_0)b]}$ is at least $\dia U$. We also notice that
$$\abs{\lr{h\mo}'(ta+(1-t)b)}\leq D\abs{h'(z)}\mo\mfa t\in[0,t_0],$$
which completes the proof as $\abs{(t_0-1)a+(1-t_0)b}\leq\dia h(U)$.
\end{proof}

Proposition 10 in \cite{MCN1} relies on inequalities \requs{Far}{MD}. We remark that they are satisfied uniformly on 
a neighborhood of $g_0$. By Koebe's \rlem{Koebe}, the definition \requ{Far} of $\ve$ does not
depend on $f$. Let us prove the uniform version of inequality \requ{MD} in $\G$.
\begin{lemma}
\label{lemPull}
There exist $M>1$, $\be_M>\al_0$ and $r_M>0$ such that for all $\ga\in[\al_0,\be_M]$ if $W$ is
a connected open with $\dia W<r_M$, $W\mo$ a connected component of $g_\ga\mo(W)$ and $x\in W\mo$ 
then
$$\dia W\mo< M\abs{g_\ga'(x)}\mo\dia W.$$
\end{lemma}
\begin{proof}
Let $\ga\in[\al_0,\be_1]$, $x\in W\mo$ and suppose
$$3\dia W\mo\leq\dst\lr{W\mo,\cri},$$
where we denote by $\cri$ the set of critical points $\Set{c_1,c_2}$. Then by Koebe's \rlem{Koebe} the
distortion is bounded by an universal constant $M_1\geq 1$ on the ball $B\lr{x,2\dia W\mo}$. 
Using \rlem{PBDia} 
\begin{equation}
\label{equBF}
\dia W\mo\leq M_1\abs{g_\ga'(x)}\mo\dia W.
\end{equation}

Let us remark some properties of the map $f_b\fd\C\C$ defined by $f_b(z)=bz^2$ for all $z\in\C$ and
$b>0$. Let $U$ be a connected open and $V=f_b(U)$. If $3\dia U>\dst\lr{U,0}$ then there exist universal
constants $M_2,M_3>1$ such that
$$
\begin{array}{rcccl}
 bM_2\mo\dia U & < & \sup\limits_{z\in U}\abs{f_b'(z)} & < & bM_2\dia U,\\
 bM_3\mo\lr{\dia U}^2 & < & \dia V & < & bM_3\lr{\dia U}^2.
\end{array}
$$
Let us also remark that using equality \requ{Cont} if $\ga\in[\al_0,\be_1]$ and $c\in \cri$ then
$$g_\ga(x)=g_\ga(c)+\frac{g_\ga''(c)}2(x-c)^2+\frac{g_\ga'''(c)}6(x-c)^3.$$
As $g_0''(c)\neq 0$ and $g_\ga(c)$, $g_\ga''(c)$ and $g_\ga'''(c)$ are continuous there exist 
$r_M>0$, $\be_M>\al_0$ and $M_4>1$ such that if $\ga\in[\al_0,\be_M]$, $\dia W<r_M$ and 
$$3\dia W\mo>\dst\lr{W\mo,\cri},$$ 
then
\begin{equation}
\label{equBC}
\begin{array}{rcccl}
 M_4\mo\dia W\mo & < & \sup\limits_{x\in W\mo}\abs{g_\ga'(x)} & < & M_4\dia W\mo,\\
 M_4\mo\lr{\dia W\mo}^2 & < & \dia W & < & M_4\lr{\dia W\mo}^2.
\end{array}
\end{equation}
The previous inequality together with inequality \requ{BF} complete the proof.
\end{proof}

We may now prove a uniform contraction result on a neighborhood of $g_0$ in $\G$. It replaces Proposition 10 in \cite{MCN1} in the proof of Theorem A.
\begin{prop}
\label{propGCE}
For any $1<\lambda_0<\lambda$ and $\theta < 1$ there exist $\be>\al_0$, $\de>0$ and $N>0$ 
such that for all $\ga\in[\al_0,\be]$, $0<R\leq\de$, $n\geq N$ and $z\in J_\ga$, the Julia set of
$g_\ga$, if $W\in\MO{Comp}g_\ga^{-n}\lr{B(z,R)}$ and there exists $x\in\overline W$ such that
$\abs{\lr{g_\ga^n}'(x)}>\la^n$ then 
\begin{equation}
\label{equL}
\MO{diam} W<\te R\la_0^{-n}.
\end{equation}
\end{prop}
\begin{proof}

Let us fix $D\in(1,\lambda/\lambda_0)$. Let $\ve\in(0,1)$
be provided by inequality (\ref{equFar}). Let also $r_M>0$ be small and $M
> 1$ provided by the \rlem{Pull}. Let $l\geq 1$ such that
\begin{equation}
2M^{j/l}D^j\lambda^{-j}\leq \theta\lambda_0^{-j}\mbox{ for all }j\geq l.
\label{equC}
\end{equation}

Let $N=2l$. There exists $r_1<r_M$ such that for all $i=1,2$, $k\var 1N$ and any
connected component $W$ of $g_0^{-k}\lr{B(c_i,4r_1)}$
$$\MO{diam}{W}\leq2\ve\MOo{dist}{W, \mathrm{Crit}}.$$
An argument similar to the proof of \rpro{Free} and the continuity of the critical points
and of the coefficients \requ{Cont} of $g_\ga$ show that there exists $b_0>\al_0$ such that
for all $\ga\in[\al_0, b_0]$, $i=1,2$ and $k\var 1N$
$$g_\ga^{-k}\lr{B(c_i,2r_1)}\se g_0^{-k}\lr{B(c_i,4r_1)}.$$
There are only a finite number of connected components of $g_0^{-k}\lr{B(c_i,4r_1)}$
for all $i=1,2$ and $k\var 1N$. Therefore by the continuity of the critical points there
exists $b_1>\al_0$ such that for all $\ga\in[\al_0,b_1]$, $i=1,2$ and $k\var 1N$ all 
connected components of $g_\ga^{-k}\lr{B(c_i,2r_1)}$ satisfy inequality \requ{Far}.

\rcor{GBS} provides $b_2>\al_0$ and $\de>0$ such that for all $\ga\in[\al_0,b_2]$, $z\in J_\ga$ 
and $k\geq 0$
$$\MO{diam}\MO{Comp} g_\ga^{-k}(B(z,\de))<\ve r_1.$$
Let us define $\be=\min\lr{\be_M,b_0,b_1,b_2}$ and fix $\ga\in[\al_0,\be]$, $z\in J_\ga$ and $n>N$.
Then 
$$\MO{diam}\MO{Comp}g_\ga^{-k}\lr{B(z,R)^{-k}}<\varepsilon r_1<r_M\mfa 0\leq k\leq n.$$
Let us also fix $W$ and $x$ as in the hypothesis. Denote $x_k=g_\ga^{n-k}(x)\in W_k=g_\ga^{n-k}(W)$
for all $k\var 0n$. 

Let $0<k_1<\ldots<k_t\leq N$ be all the integers $0\leq k\leq n$ such that $W_k$ does not 
satisfy the inequality (\ref{equFar}).
As $\ve r_1\geq \MO{diam}{W_{k_i}}$
$$r_1>\MOo{dist}{W_{k_i},\mathrm{Crit}}\mfa 1\leq i\leq t.$$ 
Then for all $1\leq i \leq t$ there exists $c\in\{c_1,c_2\}$ such that $W_{k_i}\subseteq B(c,2r_1)$. 
By the definition of $r_1$
\begin{equation}
k_{i+1}-k_i>N\mfa 1\leq i< t.
\label{equDst}
\end{equation}

We may begin estimates. For all $0<j\leq n$ with $j\neq k_i$ for all $1
\leq i \leq t$, $W_j$ satisfies the inequality (\ref{equFar}), so the distortion
on $W_j$ is bounded by $D$. Thus by \rlem{PBDia}
\begin{equation}
\MO{diam}{W_j}\leq D|g_\ga'(x_j)|^{-1}\MO{diam}{W_{j-1}}.
\label{equD}
\end{equation}
If $j=k_i$ for some $1\leq i\leq t$ we use \rlem{Pull} to obtain
\begin{equation}
\MO{diam}{W_j}\leq M|g_\ga'(x_j)|^{-1}\MO{diam}{W_{j-1}}.
\label{equP}
\end{equation}

Let us recall that $x_n=x$ with $\abs{\lr{g_\ga^n}'(x)}>\la^n$ and that $W_0=B(z,R)$ so 
$\MO{diam} W_0=2R$. If $t\geq 2$ inequality (\ref{equDst}) yields $lt\leq 2l(t-1)=N(t-1)<n$.
Consequently, as $n>2l=N$,
$$t<\frac nl.$$
Multiplying all the relations (\ref{equD}) and (\ref{equP}) for all $0<j\leq n$ we obtain
$$
\begin{array}{rcl}
\MO{diam}{W_n}& \leq & M^tD^{n-t}\abs{\lr{g_\ga^n}'(x_n)}^{-1}\MO{diam}{W_0}\\
& < & 2M^{n/l}D^n\lambda^{-n}R\\
& \leq & \theta R\lambda_0^{-n}.
\end{array}
$$
The last inequality is inequality (\ref{equC}).
\end{proof}

As a direct consequence of inequality \requ{Dst} we obtain the following corollary.
\begin{corollary}
\label{corGCE}
Assume the hypothesis of \rpro{GCE}. If there exist $$-1\leq k_1<k_2< n$$ such that 
$v\in \ol{g_\ga^{k_1+1}(W)}$ and $\ol{g_\ga^{k_2}(W)}\cap\{c_1,c_2\}\neq\es$ then 
$k_2-k_1>N$ therefore condition $n\geq N$ is superfluous.
\end{corollary}

Let us compute a diameter estimate similar to \requ{Int}.
\begin{lemma}
\label{lemNeig}
There exist $\de>0$ and $N>0$ such that for all $\ga\in[\al_0,\be_1]$, $k\geq 1$ and $x\in I_3(N)$ 
with $\iu(x)=I_3^kI_*\ldots$ where $I_*\in\left\{I_2,I_3\right\}$, the following statement holds. 
If $x\in W\se\C$ is a connected open such that $\MO{diam}g_\ga^i(W)<\de$ for all $i\var 0{k-1}$ then
$$\MO{diam}W<\la^{-k}\MO{diam}g_\ga^k(W).$$
\end{lemma}
\begin{proof}
Let us denote $x_i=g_\ga^i(x)$ and $W_i=g_\ga^i(W)$ for all $i\var 0k$. Using \rlem{DF}, 
inequalities \requ{R} and \rlem{Bou} for $\iu_1=I_3c_1$, $\iu_2=I_3c_2$ if $I_*=I_2$ and
$\iu_1=I_3c_2$, $\iu_2=I_3^\ft$ if $I_*=I_3$ there exists $N_0>0$ that does not depend on $\ga$ 
such that if $N\geq N_0$ then
$$\abs{\lr{g_\ga^k}'(x)}>\lr{\la'}^k.$$

Let $D\in\lr{1,\frac{\la'}{\la}}$ and $\ve>0$ given by inequality \requ{Far}. Using \rlem{PBDia} it 
is enough to show that $B(x_i,2\de)$ satisfies inequality \requ{Far} for all $i\var 0{k-1}$.

Let us recall that $g_\ga(y)<y$ for all $y\in(c_2,1)=I_3\sm\{1\}$. Therefore for all $i\var 0{k-1}$
$$\dst \lr{x_i,\{c_1,c_2\}}\geq\dst \lr{x_{k-1},\{c_1,c_2\}}>\dst \lr{x(I_3c_1),\{c_1,c_2\}}.$$
Let 
$$d=\min\limits_{\ga\in[\al_0,\be_1]}\MO{dist}\lr{x(I_3c_1),\{c_1,c_2\}}$$ 
and recall that $\ve$ does not depend on $\ga$. Therefore there exists 
$$\de=\frac{d}{2(1+2\ve\mo)}>0$$
such that if $\dst \lr{y,\{c_1,c_2\}}\geq d$ then $B(y,2\de)$ satisfies inequality \requ{Far}. 
\end{proof}
The following corollary admits a very similar proof. 
\begin{corollary}
\label{corNeig}
There exist $\de>0$ and $N>0$ such that for all $\ga\in[\al_0,\be_1]$, $k\geq 1$ and $x\in I_2(\max(k,N)+1),$ 
the following statement holds. 
If $x\in W\se\C$, a connected open such that $\MO{diam}g_\ga^i(W)<\de$ for all $i\var 0{k-1}$, then
$$\MO{diam}W<\la^{-k}\MO{diam}g_\ga^k(W).$$
\end{corollary}

Let us recall that all distances and diameters are considered with respect to the Euclidean metric,
as we deal exclusively with polynomial dynamics. Let us state Lemma 3 in \cite{MCN1} in this setting.
\begin{lemma}
\label{lemEMod}
Let $f$ be a polynomial, $z\in\C$ and $0<r<R$. Let $W\in\MO{Comp}f^{-1}\lr{B(z,R)}$
and $W'\in\MO{Comp}f^{-1}\lr{B(z,r)}$ with $W'\subseteq W$. If $\Deg Wf\leq\mu$ then
$$\frac{\MO{diam}{W'}}{\MO{diam}{W}} < 32\left(\frac rR\right)^{\frac 1\mu}.$$
\end{lemma}

Let us set some constants that define the telescope construction used in the proof of Theorem A. For convenience, we use the same notations as in the proof of \rthm{MCN1} in \cite{MCN1}.
Let $\mu=2$ and $\te=\frac 1232^{-\mu}$. Let $\de_0>0$ be
provided by \rcor{GNR} and $\be_0'>\al_0$, $\de_1>0$, $N_1>0$ be provided by \rpro{GCE} applied to
$\la^{\frac 12}$. Let $\de'>0$, $N_2>0$ be provided by \rlem{Neig}, $\de''>0$, $N_3>0$ be provided 
by \rcor{Neig} and $\be_M>\al_0$, $r_M>0$ and $M>1$ defined by \rlem{Pull}.

Let us observe that
$$I_1^\ft\prec I_1c_2\prec c_1\prec I_2c_2\prec \iti\prec I_2c_1\prec c_2\prec I_3c_1\prec I_3^\ft$$
and that all these sequences are continuously realized on $[\al_0,\be_1]$. Let us define
$$\ve_0=\min\limits_{\ga\in[\al_0,\be_1]}\lr{\abs{x(I_1c_2)-c_1},\abs{x(I_2c_2)-c_1},
\abs{x(I_2c_1)-c_2},\abs{x(I_3c_1)-c_2}}$$
therefore $\ve_0>0$ is smaller than $\abs{c_1-c_2}$, $\abs{c_1}$ and $\abs{1-c_2}$ for all 
$\ga\in[\al_0,\be_1]$.
We set 
\begin{equation}
\label{equEps}
\ve=\min\lr{\ve_0,\de',\de'',r_M}.
\end{equation}
 
\rcor{GBS} provides $\be_1'>\al_0$ and $\de_2>0$ such that for all $\ga\in[\al_0,\be_1']$ the diameter of any
pullback of a ball of radius at most $\de_2$ centered on $J_\ga$ is smaller than $\ve$. Let 
$\be_2'=\min\lr{\be_1,\be_0',\be_1',\be_M}$ and $\de_3=\min\lr{\de_0,\de_1,\de_2}$,
such that \rpro{GCE} applies for balls centered on $J_\ga$ of radius at most $\de_3$, for 
all $\ga\in[\al_0,\be_2']$. Moreover, \rlem{Neig} and \rcor{Neig} apply and inequalities \requs{BF}{BC} 
hold on all pullbacks of such balls. 

\rcor{GNR} provides $\de_4$ such that for 
$$r=\te R<R=\min(\de_3,\de_4),$$ 
there exist $\be_3'>\al_0$ and $N_0>0$ the time span needed 
to contract the pullback of balls of radius $R$ into components of diameter smaller than $\te R$
for all $\ga\in[\al_0,\be_3']$. We define
$$\be=\min(\be_2',\be_3').$$

Let $f$ be a rational map and $\MO{Crit}$ its critical set.
If $W\se\RS$ is an open set and $\overline{f^k(W)}$ contains at most one critical
point for all $0\leq k < n$, let us define
$$\Deg{\overline W}{f^n}=\prod_{\substack{c\in \overline{f^k(W)}\cap\mathrm{Crit} \\ 0\leq k<n}}\MO{deg}(c),$$
counted with multiplicities. 

The following fact provides the hypothesis of \rcor{GCE}. 

\begin{corollary}
\label{corDeg}
For all $\ga\in[\al_0,\be]$, $z\in J_\ga$, $0<r\leq R$ and $(W_k)_{k\geq 0}$ a backward
orbit of $B\left(z,r\right)=W_0$, if $\deg_{\overline{W_k}}g_\ga^k>\mu$ then there exist 
$0<k_1<k_2\leq n$ such that $\ol{W_{k_1}}\cap \{c_1,c_2\}\neq\emptyset$ and $c_2\in \ol{W_{k_2}}$.
\end{corollary}
\begin{proof}
By the definition of $R$, $\MO{diam}W_k<\ve\leq\ve_0<\abs{c_1-c_2}$ therefore $\ol{W_k}$ contains at 
most one critical point for all $k\geq 0$. As $\mu=\mu_{c_1}=\mu_{c_2}$ there exist $0<k_1<k_2\leq n$ 
such that $\ol{W_{k_1}}$ and $\ol{W_{k_2}}$ contain exactly one critical point each. Suppose 
$c_1\in \ol{W_{k_2}}$ therefore $1\in \ol{W_{k}}$ for all $0\leq k<k_2$ which 
contradicts $\MO{diam}W_{k_1}<\ve\leq\ve_0$.
\end{proof}

Let us prove the main result of this section.
\begin{proof}[Proof of Theorem A]
This proof has two parts. The first part describes the construction of a convergent sequence of families
$\sqnz \G$ of bimodal polynomials with negative Schwarzian derivative. Its limit $g$ does not satisfy
$RCE$. The second part shows that $g$ has $ExpShrink$ and it is similar to the
proof of \rthm{MCN1}. 

Let us recall the construction of the family $\G_1$. It is described by the common prefix $S_1$ of its
kneading sequences $\ku(\ga)$ \allg 1. We defined $S_1=I_1^{k_0+1}I_2^{k_1}$ so $\be_1<\ku\mo\lr{I_1^{k_0}c_1}$
which converges to $\al_0=0$ as $k_0\ra\ft$. Using this convergence, inequalities \requ{R}, \rlem{DF} 
applied to $v=g_\ga(c_2)$ and \rlem{Bou} applied to $\iu_1=I_1c_1$, $\iu_2=I_1c_2$ to bound $\abs{g_\ga'(v_{k_0})}$ there
exists $k_0>0$ such that the following inequalities hold
$$\be_1<\max\ku\mo\lr{I_1^{k_0}c_1}<\be,$$
$$d_\ga(k)>\la^k\mfa \ga\in[\al_0,\be_1]\mbox{ and }k\var 1{k_0+1}.$$
Again by \rlem{DF}, property \requ{Kn} and inequalities \requ{R}, if $k_1$ is sufficiently large then
\begin{equation}
\label{equCEo}
d_\ga(k)>\la^k\mfa \ga\in[\al_1,\be_1]\mbox{ and }k\var 1{t_1},
\end{equation}
where $t_1=k_0+1+k_1=\abs{S_1}$.
Let us choose $k_1$ such that the previous inequality holds and such that $t_1>N_1$ and
\begin{equation}
\label{equEat}
\max\lr{\ve R\mo,2M_4\lr{\te R}\mo,\ve^2\lr{\te R}^{-2},2M_1}<\la^{t_1-1},
\end{equation}
where $M_1$ and $M_4$ are defined by inequalities \requ{BF} respectively \requ{BC}. 
This achieves the construction of the family $\G_1$.

For all $k\geq 1$ we construct $\G_{2k}$ using \rpro{A} with 
$$\la\mo<\la_1<\la_2<1,$$
and $\G_{2k+1}$ using \rpro{B} with 
$$\De_k=2^{-k}.$$
Using inequality \requ{Conv} the sequence $\sqno \G$ converges to a map $g=g_\gl$. Let us denote 
$d(n)=d_n(\gl)=\abs{\lr{g^n}'(v)}$ and $d(n,p)=d_{n,p}(\gl)=\abs{\lr{g^p}'(v_n)}$ for all $n,p\geq 0$, 
where $v$ is the second critical value and $v_n=g^n(v)$. For all $n\geq 2$ let $p_n=p$ be provided
by \rpro{A} or \rpro{B} used to construct $\G_n$. Therefore for all $n\geq 1$
$$t_n<p_{n+1}<t_{n+1},$$
where $t_n=\abs{S_n}$ the length of the common prefix $S_n$ of kneading sequences in $\G_n$. As 
$\gl\in[\al_n,\be_n]$ for all $n\geq 1$,
$$\ku=\ku(\gl)\in S_n\ts\mfa n\geq 1.$$
Let us also recall that for all $k\geq 1$
$$S_{2k}=S_{2k-1}I_2^{k_1+1}I_3^{k_2}I_2^{k_3},$$
and that we may choose $k_1$, $k_2$ and $k_3$ as large as we need. We impose therefore 
for all $k\geq 1$
\begin{equation}
\label{equLong}
k_1>N_3,k_2>N_2\ma k_3>N_3.
\end{equation}

Let us remark that $g(c_1)=1$, $g(1)=1$ and $\abs{g'(1)}>1$ therefore $c_1\in J$ the Julia set of $g$ and
$c_1$ is non-recurrent and Collet-Eckmann. Let us remark that $\De_k\ra 0$ as $k\ra\ft$ and 
$\gl\in[\al_{2k+1},\be_{2k+1}]$ for all $k\geq 1$ therefore the second critical orbit is recurrent. 
By construction and inequality \requ{Int} the second critical orbit accumulates on $r$ and on $1$. 
Therefore $c_2\in J$ using for example a similar argument to the proof of \rlem{Acc}. Let us show that
$c_2$ is not Collet-Eckmann. Indeed, by \rpro{A} for all $k\geq 1$
$$d(p_{2k})<\la_2^{p_{2k}}<1,$$
and $p_{2k}\ra\ft$ as $k\ra\ft$. Therefore by \rdef{RCE}
$$g\mbox{ is not }RCE.$$ 
Combining inequalities \requ{CEo} and \requ{Dn}, the third claim of \rpro{A} and the second claim of \rpro{B} 
\begin{equation}
\label{equCE}
d(n)>\la^n\mfa n\in\bigcup\limits_{k\geq 0}\SetEnu{t_{2k}}{p_{2k+2}-1}.
\end{equation}
Let us check that for all $m>0$ such that $\abs{g^m(c_2)-c_2}<\ve$ 
\begin{equation}
\label{equCER}
d(m)>\la^m.
\end{equation}
Let us recall that $\ve\leq\ve_0$ by its definition \requ{Eps} so $\abs{g^m(c_2)-c_2}<\ve$ implies that
$v_m=g^{m+1}(c_2)\in I_1$ therefore $\ku(m)=I_1$ so there exists $k\geq 1$ such that 
$$t_{2k}<m<t_{2k+1},$$ 
as \rpro{A} extends $S_{2n-1}$ to $S_{2n}$ using only the symbols $I_2$ and $I_3$ for
all $n\geq 1$. Therefore $m\in\SetEnu{t_{2k}}{p_{2k+2}-1}$ thus inequality \requ{CER} is a direct 
consequence of inequality \requ{CE}.

Let us show that $g$ has $ExpShrink$. We use a telescope that is very similar to the one used in the 
proof of \rthm{MCN1}. Loosely speaking, the strategy is the following. We consider pullbacks of a small ball centered on the Julia set $J$ of $g$. We show that after some time, the pullbacks are contracted. We include such a component in another ball centered on $J$. The construction is achieved inductively and contractions at each step are used to show uniform contraction, thus $ExpShrink$. We call blocks of the telescope the sequences of pullbacks associated to each step. In order to deal with various configurations of critical points inside the telescope, we need three types of blocks.

Let us introduce additional notation and rigorously define the telescope.

For $B\subseteq \overline{\mathbb C}$ connected and $n\geq 0$, we write $B^{-n}$ or $g^{-n}(B)$ for some
connected component of $g^{-n}(B)$. When $z\in B$ and some backward
orbit $z_n\in g^{-n}(z)$ are fixed, $B^{-n}$ is the connected component
of $g^{-n}(B)$ that contains $z_n$.

We consider a pullback of an arbitrary ball
$B\left(z,R\right)$ with $z\in J$, of length $N>0$. We show that there
are constants $C_1>0$ and $\lambda_3>1$ that do not depend on $z$
nor on $N$ such that
$$\MO{diam} B\left(z,R\right)^{-N}\leq C_1\lambda_3^{-n}.$$
It is easy to check that the previous inequality for all $z\in J$ and $N>0$
implies the $ExpShrink$ condition.

Let $(z_n)_{n\geq 0}\se J$ be a backward orbit of $z$, that is $z_0=z$ and $g(z_{n+1})=z_n$ for all $n\geq 0$. 
We consider preimages of $B(z,R_0'):=B(z,R)$ up to time $N$. We show that there is some
moment $N_0'$ when the pullback $B(z,R_0')^{-N_0'}$ observes a strong contraction. Then $B(z,R_0')^{-N_0'}$
can be nested inside some ball $B(z_{N_0'},R_1')$ where $R_1'\leq R$. 
This new ball is pulled back and the construction is achieved
inductively. The pullbacks $B(z,R_0')$, $B(z,R_0')^{-1}\dots B(z,R_0')^{-N_0'}$ form
the first block of the telescope. The pullbacks $B(z_{N_0'},R_1')$,
$B(z_{N_0'},R_1')^{-1}\dots B(z_{N_0'},R_1')^{-N_1'}$ form the second block and so on.
\rlem{EMod} is essential to manage the passage between
two such consecutive telescope blocks.
We show contraction for every block using either \rcor{GBS}
or \rpro{GCE}. This leads to a classification of blocks depending
on the presence and on the type of critical points inside them.

Let $R'$ be the radius of the initial ball of some block and $N'$ be its length.
We introduce a new parameter $r'< R$ for each block, a lower bound for $R'$.
It is an upper bound of the diameter of the last pullback
of the previous block. This choice guarantees that consecutive blocks are
nested. A block that starts at time $n$ is defined by the choice of $R'$ with $r'
\leq R'\leq R$ and of $N'$ with $1\leq N'\leq N-n$. It is the pullback
of length $N'$ of $B(z_n,R')$. 

For all $n,t\geq 0$ and $r>0$ we denote
$$
\begin{array}{rcl}
d(n,r,t)& = & \Deg{B(z_n,r)^{-t}}{g^t}\mbox{ and}\\
\overline d(n,r,t)& = & \Deg{\overline{B(z_n,r)^{-t}}}{g^t}.
\end{array}
$$
Fix $n\geq 0$ and $t \geq 1$ and consider the maps $d$ and $\overline d$ defined
on $[r',R]$. They are increasing and $d\leq\overline d$. 
Moreover, for all $n\geq 0$, $r>0$, $t\geq 0$ and $s > 0$,
$$\overline d(n,r,t)\leq d(n,r+s,t).$$ 
The set ${\left\{r\in[r',R]\ |\ d(n,r,t)<\overline d(n,r,t)\right\}}
$ is the common set of discontinuities of $d$ and
$\overline d$. Note also that $d$ is lower semi-continuous and $\overline d$ is upper
semi-continuous.

For transparency, we also denote
$$W_k=B(z_n,R')^{-k}.$$

Let us define the three types of blocks. For convenience, we keep the same notations as in the proof of \rthm{MCN1} in \cite{MCN1}.
\begin{description}
\item[Type 1] Blocks with $R'=r'$ and $N'$ such that $\ol d(n,R',N')>1$ and $c_2\in\ol{W_{N'+1}}$.
\item[Type 2] Blocks with $R'=R$, $N'=\min(N_0,N-n)$ and $d(n,R,N-n) \leq \mu$.
\item[Type 3] Blocks with $\ol d(n,R',N')>1$, $c_2\in\ol{W_{N'+1}}$ and $d(n,R',N-n) \leq \mu$.
\end{description}

Let us define $r'$. It is the diameter of the last pullback of the previous block of type 1 or 3. It is $r=\te R$ if the previous block is of type 2 and $r'=R$ for the first block.

Let us first show that for all $z\in J$ and $N>0$ we may define a telescope using the three types of blocks. The construction is inductive and we show that given $0\leq n<N$ and $0<r'\leq R$ we can find $R'\in[r',R]$ and
$0<N'\leq N-n$ that define a block of one of the three types. We also show contraction along every block so $r'$ defined as above is smaller than $R$, thus completing the proof of the existence of the telescope.

If $\overline d(n,r',N -n +1)>\mu=2$ then by \rcor{Deg} there is $1\leq N' \leq N-n$ that defines a type 1 pullback for $R'=r'$.
If $\overline d(n,R,N-n)\leq\mu$ then we define a type 2 block, as $d\leq\overline d$. 
Note that the first block of the telescope is already constructed as $r'=R$. In all other cases $r'<R$.
If $\overline d(n,R,N-n+1) > \mu$ there is a smallest $R'$, with $r'<R'\leq R$, such that
$\overline d(n,R',N-n+1) > \mu$. Thus $R'$ is a point of discontinuity of $\overline d$
so $d(n,R',N-n+1)<\overline d(n,R',N-n+1)$, therefore $d(n,R',N-n) \leq d(n,R',N-n+1) \leq \mu$.
Then by \rcor{Deg} there is $1\leq N' \leq N-n$ that defines a type 3 pullback.

Let us be more precise with our notations. We denote by $n_i'$, $N_i'$, $r_i'$
and $R_i'$ the parameters $n$, $N'$, $r'$ and $R'$ of the $i$-th block. Let also $W_{i,k}$
be $W_k$ in the context of the $i$-th block with $i\in\{0,\ldots, b\}$, where $b+1$
is the number of blocks of the telescope. So $n'_0=0$, $r_0'=R$ and $n_1'=N_0'$.
In the general case $i>0$, we have
$$
\begin{array}{rcl}
n_i' & = & n'_{i-1} + N'_{i-1}\mbox{ and}\\
r_i' & \geq & \MO{diam} W_{i-1,N'_{i-1}}.
\end{array}
$$
Let us also denote by $T_i\in\{1,2,2',3\}$ the type of the $i$-th block. The type
$2'$ is a particular case of the second type, when $N'<N_0$. This could only
happen for the last block, when $N-n'_b<N_0$. So $T_i\in\{1,2,3\}$ for all
$0\leq i<b$. We may code our telescope by the type of its blocks, from right to left
$$T_b\ldots T_2T_1T_0.$$

Our construction shows that
\begin{equation}
\MO{diam} W_{i-1,N_{i-1}'}\leq r'_{i}<R\mbox{ for all }0< i < b
\label{condSuf}
\end{equation}
is a sufficient condition for the existence of the telescope that contains the pullback of $B(z,R)$ of length $N$. 

If $T_i=2$ we apply \rcor{GBS} so
\begin{equation}
\MO{diam} W_{i,N_i'}<r'_{i+1}=\theta R<R.
\label{estNR}
\end{equation}

If $T_i\in\{1,3\}$ we show that there exists $\la_0>1$ such that
\begin{equation}
\MO{diam} W_{i,N_i'}<\theta R_i'\lambda_0^{-N_i'}<R,
\label{estCE}
\end{equation}
as $\theta<\frac 12$, $R_i'\leq R$ and $\lambda_0^{-N_i'}<1$. 
This inequality completes the proof of the existence of the telescope.
 
Let us fix $i\geq 0$ such that $T_i\in\Set{1,3}$. Suppose that $i>0$ and $T_{i-1}\in\Set{1,3}$ also, 
therefore 
$$c_2\in \ol{W_{i-1,N_{i-1}'+1}}\se \ol{W_{i,1}}=\ol{B(z_{n_i},R_i')\mo}=
\ol{g^{N_i'}\lr{W_{i,N_i'+1}}}\se \ol{B(z_{n_i},R)\mo}.$$
But $c_2\in \ol{W_{i,N_i'+1}}$ also and $\dia B(z_{n_i},R)\mo<\ve$ by the definition of $R$. 
Therefore by inequality \requ{CER}
$$d(N_i')>\la^{N_i'}$$
so by \rcor{GCE} we may apply \rpro{GCE} to obtain
$$\dia W_{i,N_i'}<\te R_i'\la^{-\frac{N_i'}2}.$$
We have proved that for all $i>0$ with $T_{i-1}\in\Set{1,3}$ inequality (\ref{estCE}) holds for all 
$\la_0\leq \la^\frac 12$. If $i=0$ or $T_{i-1}=2$ then $R_i'\in[\te R,R]$. 
Therefore it is enough to show that there exists $\la_0>1$ such that for all $z\in J$, $\tau\in[\te R,R]$, $n>0$ and $W$ a connected component of $g^{-n}\lr{B(z,\tau)}$ the following statement holds. If $v\in \ol W$ and there
exist $0\leq m<n$ such that $\ol{g^m(W)}\cap\cri\neq\es$ then
\begin{equation}
\label{equES}
\dia W<\te \tau\la_0^{-n}.
\end{equation}
Again, if $d(n)>\la^n$ using \rcor{GCE} and \rpro{GCE} the previous inequality is satisfied for all 
$1<\la_0\leq\la^\frac 12$. Therefore using inequality \requ{CE} we may suppose that there exist 
$k'\geq 1$ such that
$$p_{2k'}\leq n<t_{2k'}.$$

Let us denote $p=p_{2k'}$, $t=t_{2k'-1}$ and $W^k=g^k(W)$ for all $k\var 0n$. By the definition of
$p$ in \rpro{A} 
\begin{equation}
\label{equTP}
2t<p.
\end{equation}
\def\dc{\dia W^{p-1}}
\def\dv{\dia W^{p}}
Using \rcor{Neig}, inequalities \requs{Long}{Eat}
$$\dw t<\lmp{p-1-t}\dc<\lmp{p-1-t}\ve<R.$$
As $t_1>N_1$, inequality \requ{Dn} lets us apply \rpro{GCE} to $B\lr{v_t,\dw t}$ which combined
to the previous inequality shows that 
\begin{equation}
\label{equLeft}
\dia W<\te\lmp{p-1-\frac t2}\dc.
\end{equation}

By inequalities \requ{Long}, using \rlem{Neig} and eventually \rcor{Neig} if $v_n\in I_2$ 
\begin{equation}
\label{equRight}
\dv<\lmp{n-p}\dw n=2\lmp{n-p}\tau.
\end{equation}

Therefore the only missing link is an estimate of $\dc$ with respect to $\dv$. We distinguish the
following two cases.
\begin{enumerate}
\item $\dst\lr{W^{p-1},c_1}<3\dc$.
\item $\dst\lr{W^{p-1},c_1}\geq 3\dc$.
\end{enumerate}

Suppose the first case. The by the definition \requ{Eps} of $\ve$ we may use inequality \requ{BC}
therefore
$$
\begin{array}{rcl}
 \dc & < & \lr{M_4\dv}^\frac 12 \\
     & < & \lr{2M_4\tau}^\frac 12\lm{\frac{n-p}2}
\end{array}
$$
using inequality \requ{Right}. Recall that $\tau\geq \te R$ and $t\geq t_1$. Therefore \ineqss{Left}{TP} 
and \requ{Eat} imply that
$$
\begin{array}{rcl}
 \dia W & < & \te\lm{\frac n2}\tau\lr{2\la M_4r\mo}^\frac 12\lm{\frac t2}\\
        & < & \te\lm{\frac n2}\tau.
\end{array}
$$
Therefore in the first case it is enough to choose $\la_0\leq\la^\frac 12$.

Suppose the second case. Using \ineqss{Left}{TP} and \requ{Eat} we may compute 
\begin{equation}
\label{equBrut}
\begin{array}{rcl}
 \dia W & < & \te\lmp{p-1-\frac t2}\ve\\
        & < & \te\lm{\frac p2}\te R\leq\te\lm{\frac p2}\tau\\
        & = & \te\lm{n\lr{\frac p{2n}}}\tau.
\end{array}
\end{equation}
This is not enough as $\la_0$ should depend only on $g$. We may remark that we are in position to use
\ineq{BF} for $W^p$ therefore
$$\dc<M_1\abs{g'(v_{p-1})}\mo\dv.$$
Let us compute an upper bound for $\abs{g'(v_{p-1})}\mo=d(p-1,1)\mo$. Using the first two claims of
\rpro{A}
$$d(p)\mo=d(p-1)\mo d(p-1,1)\mo<\la_1^{-p}<\la^p,$$
and
$$d(p-1)<\la_r^{p-1}\la^{p-1},$$
where we denote $\la_r=\abs{g'(r(\gl))}$. Let $\nu=\frac{\log \la_r}{\log \la}$. Combining the
previous inequalities
$$d(p-1,1)\mo<\la^{p(\nu+2)},$$
therefore using \ineqss{Left}{Right}, \requ{TP} and \requ{Eat}
$$
\begin{array}{rcl}
 \dia W & < & 2M_1\te\lmp{p-1-\frac t2}\la^{p(\nu+2)}\lmp{n-p}\tau \\
        & < & \te(2M_1)\la^{\nu p+2p+1+\frac t2}\lm n\tau \\
        & < & \te\lm{n+p(\nu + 3)}\tau.
\end{array}
$$
If $n>2p(\nu+3)$ then \ineq{ES} is satisfied for all $\la_0\leq \la^\frac 12$. If $n\leq 2p(\nu+3)$
then using \ineq{Brut}, \ineq{ES} is satisfied for all 
$$\la_0\leq \la^\frac 1{4(\nu+3)}\leq \la^\frac p{2n},$$
which completes the proof of inequality (\ref{estCE}) and therefore of the existence of the telescope.

The remainder of the proof shows the global exponential contraction of the diameter of pullbacks. It is identical to the second part of the proof of \rthm{MCN1} in \cite{MCN1}. We reproduce it here for convenience.

Note that if $T_i=1$ then we may rewrite inequality (\ref{estCE}) as follows
\begin{equation}
r'_{i+1}<\theta r_i'\lambda_0^{-N_i'}<r_i'\lambda_0^{-N_i'}.
\label{estTO}
\end{equation}
Recall also that if there are $\la_3>1$ and $C_1>0$ such that
\begin{equation}
\MO{diam} B(z,R)^{-N} = \MO{diam} W_{0,N}<C_1\la_3^{-N},
\label{estCO}
\end{equation}
then the theorem holds. We may already set
\begin{equation}
\la_3=\min\left(2^{\frac {1}{\mu N_0}},\lambda_0^{\frac {1}{\mu}}\right).
\label{valLO}
\end{equation}

As inequality (\ref{estTO}) provides an easy way to deal with the first type of block,
we compute estimates only for sequences of blocks of types $1\ldots 1$, $1\ldots 12$ and $1\ldots 13$,
as the sequence $T_b\ldots T_2T_1T_0$ can be decomposed in such sequences. Sequences with only one block
of type $2$ or $3$ are allowed as long as the following block is not of type $1$.
For a sequence of blocks $T_{i+p}\ldots T_i$, let
$$N'_{i,p}=N'_{i+p} + \ldots + N'_i$$
be its length.

A sequence $1\ldots 1$ may only occur as the first sequence of blocks, thus $i=0$. As $r'_0=R$, iterating inequality (\ref{estTO}) for such a sequence we obtain

\begin{equation}
\begin{array}{rcl}
r'_{p+1} &   <  & \theta^{p+1} R\lambda_0^{-N'_{0,p}}\\
         &   <  & 2\theta R'_0\la_3^{-\mu N'_{0,p}}.
\end{array}
\label{estS1}
\end{equation}

Combining inequalities (\ref{estTO}), (\ref{estNR}) and the definition (\ref{valLO}) of $\la_3$,
for a sequence $1\ldots 12$
\begin{equation}
\begin{array}{rcl}
r'_{i+p+1} &   <  & r'_{i+1}\lambda_0^{-N'_{i+1,p-1}}\\
           &   <  & 2\theta R2^{-1}\lambda_0^{-N'_{i+1,p-1}}\\
           & \leq & 2\theta R'_i\la_3^{-\mu N'_{i,p}},
\end{array}
\label{estSS}
\end{equation}
as $N_i'=N_0$, $N'_{i,p}=N_i'+N'_{i+1,p-1}$ and $R'_i=R$.

For a sequence $1\ldots 13$, inequalities (\ref{estTO}) and (\ref{estCE}) yield
\begin{equation}
\begin{array}{rcl}
r'_{i+p+1} &   <  & r'_{i+1}\lambda_0^{-N'_{i+1,p-1}}\\
           &   <  & \theta R'_i\lambda_0^{-N'_{i,p}}\\
           &   <  & 2\theta R'_i\la_3^{-\mu N'_{i,p}}.
\end{array}
\label{estST}
\end{equation}

We also find a bound for $r_{b+1}'$ in the case $T_b=2'$. 
Using notations introduced in Section \ref{sectRat} we define $K=||g'||_{\infty,J_{+\ve}}$. We compute 
\begin{equation}
\begin{array}{rcl}
r'_{b+1} & < & R'_b\mu K^{N'_b}\\
         & = & \mu (K\la_3)^{N'_b} R'_b\la_3^{-N'_b}.
\end{array}
\label{estSP}
\end{equation}

We decompose the telescope into $m+1$ sequences $1\ldots 1$, $1\ldots 12$, $1\ldots 13$ and
eventually $2'$ on the leftmost position
$$S_m\ldots S_2S_1S_0.$$

Consider a sequence of blocks
$$S_j=T_{i+p}\ldots T_i.$$
Denote $n''_j=n'_i$, $N''_j=N'_{i,p}$, $r''_j=r'_i$ and $R''_j=R'_i$. Let also
$$\Delta_j=\MO{diam} W_{i,N-n'_i}$$
 be the diameter of the pullback of the first block
of the sequence up to time $-N$.

With the eventual exception of $S_m$, inequalities (\ref{estS1}), (\ref{estSS}) and  (\ref{estST})
provide good contraction estimates for each sequence $S_j$
$$r''_{j+1}<2\theta R''_j \la_3^{-\mu N''_j}.$$
If $T_b=2'$ then inequality (\ref{estSP}) yields a constant $\mu (K\la_3)^{N'_b}<C_1=\mu (K\la_3)^{N'_0}$
such that
$$r''_{m+1}<C_1 R''_m\la_3^{-N''_m},$$
as $R_m''=R_b'$ and $N_m''=N_b'$.
Note that the previous inequality also holds if $T_b\in\{1,2,3\}$.
We cannot simply multiply these inequalities as $R''_j > r''_j$ for all $0<j\leq m$.

By the definitions of types $2$ and $3$, the degree $d(n''_j,R''_j,N-n''_j)$
is bounded by $\mu$ in all cases. So \rlem{EMod} provides a bound for
the distortion of pullbacks up to time $-N$

$$
\begin{array}{rcl}
\frac {\Delta_{j-1}}{\Delta_j}& < & 32\left(\frac{r''_j}{R''_j}\right)^{\frac 1\mu}\\
& < & 32\left(2\theta  \la_3^{-\mu N''_{j-1}}\frac{R''_{j-1}}{R''_j}\right)^{\frac 1\mu}\\
& = & \la_3^{-N''_{j-1}}\left(\frac{R''_{j-1}}{R''_j}\right)^{\frac 1\mu},
\end{array}
$$
for all $0< j \leq m$. Therefore
$$\frac {\Delta_0}{\Delta_{m}} < \la_3^{-N + N''_m}\left(\frac{R''_0}{R''_m}\right)^{\frac 1\mu}.$$
Recall that $R''_j \leq R < 1$ for all $0\leq j \leq m$ and $\Delta_{m}=r''_{m+1}$, so
$$
\begin{array}{rcl}
\Delta_0 & < & \la_3^{-N + N''_m} C_1 R''_m\la_3^{-N''_m} \left(\frac{R}{R''_m}\right)^{\frac 1\mu}\\
& < & \la_3^{-N}C_1 (R''_m)^{1-\frac 1\mu}\\
& < & C_1\la_3^{-N}.
\end{array}
$$
By definition $\Delta_0=\MO{diam} W_{0,N}$, therefore the previous inequality combined with inequality
(\ref{estCO}) completes the proof of the theorem.
\end{proof}

\section{RCE is not a topological invariant for real polynomials with negative Schwarzian derivative}
\label{sectRCEnTOP}
Let $\Hc\fd{[0,h]}{\Pc_2}$ (see the definition of $\Pc_2$ at page \pageref{defP2}) be equal to the 
family $\G$ defined in the previous section. Let us define
another family of bimodal maps $\Ht\fd{[0,h']}{\Pc_2}$ in an analogous fashion. Let $T\in\R_7[x]$ be
a degree $7$ polynomial such that $T(0)=0$ and such that $T'(x)=(x+1)^3(x-1)^3$. Therefore $T$ has
two critical points $-1$ and $1$ of degree $4$ and $T(-x)=-T(x)$ for all $x\in\R$. 
Let $y_0=T(-1)$ and $x_0>1$ such that $T(x_0)=y_0$.
Let $h'>0$ be small and for each $\ga\in[0,h']$ two order preserving linear maps 
$R_{\ga'}(x)=x(2x_0+\ga')-x_0-\ga'$ and $S_{\ga'}(y)=\frac{y-T(-x_0-\ga')}{y_0-T(-x_0-\ga')}$ that map
$[0,1]$ onto $[-x_0-\ga',x_0]$ respectively $[T(-x_0-\ga'),T(x_0)]$ onto $[0,1]$. One may show by direct
computation that if a real polynomial $P$ is such that all the roots of $P'$ are real then $P$ has negative
Schwarzian derivative. Therefore 
$$\ti h_{\ga'}=S_{\ga'}\circ T\circ R_{\ga'}\in\Pc_2\mfa \ga'\in[0,h'].$$
We define $\Ht(\ga')=\ti h_{\ga'}$ for all $\ga'\in[0,h']$. 
Let us remark that $y_0\in(0,1)$ and $x_0\in\lr{\frac 32,2}$ therefore all three fixed points of $\ti h_0$ are repelling.
Let $\ti r(\ga')$ be the only fixed point of $\ti h_{\ga'}$ in $(0,1)$ and $\ti c_1<\ti c_2$ its critical 
points. The proofs that for $h'>0$ sufficiently small $\Ht$ satisfies properties (\ref{equGRep}) to 
(\ref{equGKn}), \rlems{Fix}{Per}, that it is natural, that $\ti r$, $\ti c_1$ and $\ti c_2$ 
are continuous and that for all $n>1$
$$\ti r(\ga')-\ti h_{\ga'}^n(\ti c_2)\mbox{ has finitely many zeros in }[0,h'],$$
go exactly the same way as for $\G$. As $h_0'(r(0))=-3$, $h_0'(1)=9$, $y_0=\frac{16}{35}$, $\frac 32<x_0<2$,
$\ti r(0)=0$ and $\abs{h_0'(\ti r(0))}=\frac{x_0}{y_0}$
one may compute that 
$$\frac 12\frac{\log \abs{h_0'(1)}}{\log \abs{h_0'(r(0))}}=1<
\frac 34\frac{\log \abs{\ti h_0'(1)}}{\log \abs{\ti h_0'(\ti r(0))}}.$$
We may also suppose $h>0$ and $h'>0$ sufficiently small such that there exist $1<\la<\la'$,
$1<\ti\la<\ti\la'$ and $\te_1<\te_2$ such that for all $\ga\in[0,h]$ and $\ga'\in[0,h']$
$$
\begin{array}{rcl}
 \la' & < & \min\lr{\abs{h_\ga'(0)},\abs{h_\ga'(r)},\abs{h_\ga'(1)}} \ma\\
 \ti\la'& < & \min\lr{\abs{\ti h_{\ga'}'(0)},\abs{\ti h_{\ga'}'(\ti r)},\abs{\ti h_{\ga'}'(1)}}
\end{array}
$$ and
\begin{equation}
\label{equLog}\frac 12\frac{\log \abs{h_\ga'(1)}}{\log \abs{h_\ga'(r(\ga))}}<\te_1<\te_2<
\frac 34\frac{\log \abs{\ti h_{\ga'}'(1)}}{\log \abs{\ti h_{\ga'}'(\ti r({\ga'}))}}.
\end{equation}

\def\inn{[\al_n',\be_n']}
Let us denote $\ku(\ga)$ the second kneading sequence of $h_\ga$ and $\ti\ku(\ga')$ the second kneading
sequence of $\ti h_{\ga'}$. We construct two decreasing sequences of families of bimodal maps
$(\Hc_n)_{n\geq 1}$ and $(\Ht)_{n\geq 1}$. Let $\Hc_n\fd {[\al_n,\be_n]}\Pc_2$ with 
$\Hc_n(\ga)=\Hc(\ga)$ \allg n and
$\Ht_n\fd\inn\Pc_2$ with $\Ht_n(\ga)=\Ht(\ga)$ \allgp n. By construction we choose that for all $n\geq 1$
$$\ku(\al_n)=\ti\ku(\al_n')\ma \ku(\be_n)=\ti\ku(\be_n').$$
Let us denote $v=h_\ga(c_2)$, $\ti v=\ti h_{\ga'}(\ti c_2)$ and $v_n=h_\ga^n(v)$, 
$\ti v_n=\ti h_{\ga'}^n(\ti v)$ for
all $n\geq 0$, $\ga\in[0,h]$ and $\ga'\in[0,h']$. Let also $d_n(\ga)=\abs{\lr{h_\ga^n}'(v)}$, $\ti d_n(\ga')=
\abs{\lr{\ti h_{\ga'}^n}'(\ti v)}$, $d_{n,p}(\ga)=\abs{\lr{h_\ga^p}'(v_n)}$ and $\ti d_{n,p}(\ga')=
\abs{\lr{\ti h_{\ga'}^p}'(\ti v_n)}$ for all $n,p\geq 0$, $\ga\in[0,h]$ and $\ga'\in[0,h']$.
The basic construction tool is again \rpro{Conv} and we build the sequences $(\Hc_n)_{n\geq 1}$ and 
$(\Ht)_{n\geq 1}$ by specifying the common prefix $S_n$ of the kneading sequences in $\Hc_n$ and 
$\Ht_n$ for all $n\geq 1$. We also reuse the notation $t_n=|S_n|$ for all $n\geq 1$. In an analogous way
to the construction of the family $\G_1$, see inequality \requ{CEo}, we choose
$$S_1=I_1^{k_0+1}I_2^{k_1}$$
such that  
\begin{equation}
\label{equHCEo}
d_k(\ga)>\la^k\ma \ti d_k({\ga'})>\ti \la^k
\end{equation}
\allg 1, $\gabp 1$ and $k\var 1{t_1}$ and
$$\be_1<h\ma \be_1'<h'.$$
 
Let us describe the construction of the sequences $(\Hc_n)_{n\geq 1}$ and $(\Ht)_{n\geq 1}$ which satisfy
properties \requ{Dn} to \requ{Si} and  
\begin{equation}
\label{equDnp}
\ti d_{t_n}(\ga')>\ti\la^{t_n}
\end{equation}
for all $n\geq 1$.
 
Let us recall that \rpro{A} employs twice \rpro{Conv} to construct a subfamily $\G_{n+1}$ of $\G_n$ with
$$S_{n+1}=S_nI_2^{k_1+1}I_3^{k_2}I_2^{k_3}.$$ 
Let $\gt$ and $\gtp$ be provided by \rpro{Conv} such that $\ku(\gl)=\ti\ku(\gl')=S_n\iti$. We use the 
same strategy as in the proof of \rpro{A} to define both $\Hc_{n+1}$ and $\Ht_{n+1}$ with the same combinatorics.
Taking $k_1$, $k_2$ and $k_3$ sufficiently large we may control the growth of $d_m(\ga)$ and $\ti d_m(\ga')$ 
uniformly for all $t_n<m\leq t_{n+1}$. We let 
$$\frac{k_1}{k_2}\ra\eta>0,$$  
$p=t_n+k_1+1$ and compute some bounds for $d_p(\ga)$ and $\ti d_p(\ga')$. For transparency, let us denote 
$\la_0=\abs{h_{\gt}'(r)}$, $\ti\la_0=\abs{\ti h_{\gtp}'(\ti r)}$, $\la_3=\abs{h_{\gt}'(1)}$ and
$\ti\la_3=\abs{\ti h_{\gtp}'(1)}$. As in the proof of \rpro{A} we obtain 
\begin{equation}
\label{equLLog}
\lim\limits_{k_1\ra\ft}\frac 1{k_1}\log d_p(\ga)=\log \la_0-\frac 1{2\eta}\log\la_3\ag{n+1}.
\end{equation}
We may observe that inequalities \requ{Pow} hold exactly when $c_1$ is a second degree critical point.
We may however write similar bounds for $\Ht_{n+1}$. By the same arguments there exist constants
$\ti M>1$, $\ti \de>0$ and $\ti N_2>0$ such that if $k_1>\ti N_2$ and $\ga'\in[\ga_1',\ga_2']$ then
$$
\begin{array}{lcccl}
 \ti M^{-1}(x-\ti c_1)^4 & < & |1-\ti h_{\ga'}(x)| & < & \ti M(x-\ti c_1)^4 \mbox{ and} \\
 \ti M^{-1}(x-\ti c_1)^3 & < & \abs{\ti h_{\ga'}'(x)} & < & \ti M(x-\ti c_1)^3
\end{array}
$$
for all $x\in(\ti c_1-\ti\de,\ti c_1+\ti\de)$, where $\ga_1'$, $\ga_2'$ are the bounds for
$\ga'$ provided by \rpro{Conv} applied to $S_n$ and $\Ht_n$. Therefore we obtain  
\begin{equation}
\label{equLLogp}
\lim\limits_{k_1\ra\ft}\frac 1{k_1}\log \ti d_p(\ga')=\log \ti\la_0-\frac 3{4\eta}\log\ti\la_3\agp{n+1}.
\end{equation}
Using inequalities \requ{Log} and the limits \requ{LLog} and \requ{LLogp} it is enough to choose 
$$\te_1<\eta<\te_2$$ 
to obtain the following corollary of \rpro{A}.
\begin{corollary}
\label{corA}
There exist 
$$0<\la_1<1<\la_2<\min\lr{\la,\ti\la}$$
that depend only on $\Hc_1$ and $\Ht_1$ such that if $\Hc_n$ is a subfamily of $\Hc_1$ and
$\Ht_n$ is a subfamily of $\Ht_1$ both satisfying conditions \requ{Dn} to \requ{Si} and \requ{Dnp} then
there exist $\Hc_{n+1}$ a subfamily of $\Hc_n$ and  $\Ht_{n+1}$ a subfamily of $\Ht_n$ 
satisfying the same condition and $2t_n<p<t_{n+1}$ with the following properties
\begin{enumerate}
\item $d_p(\ga)>\la_2^p$ \allg{n+1}.
\item $\ti d_p(\ga')<\la_1^p$ \allgp{n+1}.
\item $d_{t_n,l}(\ga)>\la^l$ \allg{n+1} and $l\var 1{p-1-t_n}$.
\item $\ti d_{t_n,l}(\ga')>\ti\la^l$ \allgp{n+1} and $l\var 1{p-1-t_n}$.
\item $d_{p,l}(\ga)>\la^l$ \allg{n+1} and $l\var 1{t_{n+1}-p}$.
\item $\ti d_{p,l}(\ga')>\ti\la^l$ \allgp{n+1} and $l\var 1{t_{n+1}-p}$.
\end{enumerate}
\end{corollary}

\rpro{B} has an immediate corollary for the families $\Hc$ and $\Ht$.
\begin{corollary}
\label{corB}
Let the subfamilies $\Hc_n$ and $\Ht_n$ of $\Hc_1$ respectively $\Ht_1$ with $n\geq 1$ satisfy 
conditions \requ{Dn} to \requ{Si} and \requ{Dnp} and
$$\De>0.$$
Then there exist subfamilies $\Hc_{n+1}$ of $\Hc_n$ and $\Ht_{n+1}$ of $\Ht_n$ satisfying the 
same conditions and such that there exists $t_n<p<t_{n+1}$ with the following properties
\begin{enumerate}
\item $\abs{h_\ga^p(c_2)-c_2}<\De$ \allg{n+1}.
\item $\abs{\ti h_{\ga'}^p(\ti c_2)-\ti c_2}<\De$ \allgp{n+1}.
\item $d_{t_n,l}(\ga)>\la^l$ \allg{n+1} and $l\var 1{t_{n+1}-t_n}$.
\item $\ti d_{t_n,l}(\ga')>\ti\la^l$ \allgp{n+1} and $l\var 1{t_{n+1}-t_n}$.
\end{enumerate}
\end{corollary}

For all $k\geq 1$ we define $\Hc_{2k}$ and $\Ht_{2k}$ using \rcor{A} and $\Hc_{2k+1}$ and $\Ht_{2k+1}$ 
using \rcor{B} with $\De=2^{-k}$. Let $h$ be the limit of $(\Hc_n)_{n\geq 1}$ and $\ti h$ be the limit 
of $(\Ht_n)_{n\geq 1}$. Then $h$ is $CE$ therefore $RCE$ and the second critical point $\ti c_2$ of
$\ti h$ is recurrent but not $CE$ therefore $\ti h$ is not $RCE$. Both $h$ and $\ti h$ have negative
Schwarzian derivative and their second critical orbits accumulate on $r$ and $1$ respectively on $\ti r$ and
$1$. Moreover, using \rlem{Acc}, $h$ and $\ti h$ do not have attracting or neutral periodic points on $[0,1]$. 
We may therefore apply \rcors{Hom}{Equiv} to obtain the following theorem that contradicts Conjecture 1
in \cite{TRCE}.

\begin{thmb}
The $RCE$ condition for S-multimodal maps is not topologically invariant.
\end{thmb}

\bigskip
\textbf{Acknowledgments.} The author would like to thank Jacek Graczyk who ask the question wether $RCE$ is equivalent to $TCE$. He also suggested that the construction developed for the proof of Theorem A could also be employed to prove Theorem B. The author is also grateful to Neil Dobbs, who helped improve the presentation of the paper. Part of this work was done at University of Orsay, France.


\bibliographystyle{plain}

\end{document}